\documentclass[a4paper,11pt]{article}

\usepackage[utf8]{inputenc}
\usepackage{amsmath, amsthm, amsfonts, amssymb}
\usepackage{stmaryrd}
\usepackage{graphicx}
\usepackage{xcolor}
\usepackage{paralist}
\usepackage{soul}
\usepackage{framed}
\usepackage{setspace}
\usepackage{tikz-cd}
\usepackage{bbm}
\usepackage{algorithm}
\usepackage{algorithmic}
\usepackage{hyperref}

\renewcommand{\>}{\rangle}

\newcommand{\eps}{\varepsilon}
\renewcommand{\phi}{\varphi}

\newcommand{\N}{\mathbb{N}}
\newcommand{\Z}{\mathbb{Z}}
\newcommand{\R}{\mathbb{R}}
\newcommand{\C}{\mathbb{C}}
\newcommand{\II}{\mathbb{I}}
\newcommand{\E}{\mathbb{E}}

\newcommand{\co}{\colon}

\newcommand{\bbrackets}[1]{\!\left(\rule{0pt}{11pt} #1\right)}

\newtheorem{theorem}{Theorem}
\newtheorem{lemma}[theorem]{Lemma}
\newtheorem{coro}[theorem]{Corollary}
\newtheorem{prop}[theorem]{Proposition}
\theoremstyle{definition}
\newtheorem{remark}[theorem]{Remark}
\newtheorem{example}[theorem]{Example}

\DeclareMathOperator{\Span}{span}
\DeclareMathOperator{\Tr}{Tr}
\DeclareMathOperator*{\argmin}{arg\,min}

\title{
Constructive discretization and approximation in reproducing kernel Hilbert spaces
}
\author{Abdellah Chkifa, Matthieu Dolbeault, David Krieg, Mario Ullrich}

\begin{document}

\maketitle
\begin{abstract} 
We generalize the sparsification algorithm of
 Batson, Spielman and Srivastava, 
making one part of the result dimension-independent. 
In particular, we recover discretization inequalities in $L_2$- and sup-norms 
on general finite-dimensional subspaces, prove a suitable infinite-dimen\-sional variant, 
and discuss the implications for the error of least-squares approximation based on samples.
This gives a more constructive version of several recently established approximation bounds, some of which relied on the stronger and less constructive result of Marcus, Spielman and Srivastava. 
We also improve the constants and oversampling factors in these results.
\end{abstract}

\section{Introduction}
\label{sec:intro}
In recent years, much
progress has been made in norm discretization and sampling recovery, first
by considering concentration inequalities for random points,
then by applying a sample sparsification using \cite{MSS} or \cite{BSS}. 
This follows the trend observed in other fields, including graph sparsification, frame discretization, matrix
sketching, subspace selection, and more generally randomized numerical linear algebra.

Here, we prove a generalization of the result from~\cite{BSS} by Batson, Spielman, and Srivastava (BSS), and expand on its implications
for the problems of norm discretization and sampling recovery. 
The original result from~\cite{BSS} can be stated as follows:

\begin{prop}[\cite{BSS}, Theorem 3.1]
\label{thm:BSS}
Let $D$ be a finite set and let $\mu$ be the uniform distribution on $D$. Let $a=(a_1,\hdots,a_m)^\top$
be an orthonormal family in $L_2(D,\mu)$.
Then, for any $n>m$, there exist points $x_1,\dots,x_n\in D$ and weights ${w_1,\dots,w_n>0}$ such that
\[
\lambda_{\min}\left( \sum_{i=1}^n w_i\, a(x_i)a(x_i)^* \right) \geq
\left(1-\sqrt{\frac{m}{n}}\right)^2
\]
and
\[
\lambda_{\max}\left(\sum_{i=1}^n w_ia(x_i)a(x_i)^* \right) \leq \left(1+\sqrt{\frac{m}{n}}\right)^2.
\]
\end{prop}

Here and in the following, we denote by 
$\lambda_{\min}(A)$ and $\lambda_{\max}(A)$ the smallest and largest eigenvalues of $A$.
The points and weights from Proposition~\ref{thm:BSS} can be constructed in polynomial time.

In \cite{BSS}, this theorem is applied to the problem of graph sparsification, see also \cite{HGN+24,SS11,ST11}.
For our applications, we want to generalize the result in three ways: by allowing general measure spaces $(D,\mu)$; by decoupling the lower and upper frame bounds with two different
families of $L_2$-functions; and by making the upper frame bound independent of the dimension, thus allowing to consider very large and even infinite-dimensional families. 

The first two kinds of generalizations are already available in the literature. General probability spaces are treated, for instance, in \cite[Theorem~6.4]{DPSTT}.
The second kind of generalization was achieved in \cite{BDM14}.
The result of \cite{BDM14}, called dual set spectral sparsification, can be reformulated as follows.

\begin{prop}[\cite{BDM14}, Lemma 13]
\label{prop:BDM}
Let $D$ be a finite set and let $\mu$ be the uniform distribution on $D$. Let $a=(a_1,\hdots,a_m)^\top$ and $b=(b_1,\dots,b_M)^\top$
be orthonormal families in $L_2(D,\mu)$.
Then for any integer $n>m$, there exist points $x_1,\dots,x_n\in D$ and weights ${w_1,\dots,w_n>0}$ such that
\[
\lambda_{\min}\left( \sum_{i=1}^{n} w_i\, 
a(x_i)a(x_i)^* \right) \geq
\left(1-\sqrt{\frac{m}{n}}\right)^2
\]
and
\[
\,
\lambda_{\max}\left(\sum_{i=1}^{n} w_i\, b(x_i)b(x_i)^* \right) \leq \left(1+\sqrt{\frac{M}{n}}\right)^2.
\]
\end{prop}

Our main contribution, stated below, is the third kind of generalization. Namely, we replace the dimension $M$ by an effective dimension in the
upper frame bound.
\smallskip

\begin{theorem}
\label{thm:main}
Let $(D,\mu)$ be a measure space, and let $a=(a_1,\hdots,a_m)^\top$
and $b= (b_k)_{k\in \mathbb I}^\top$ be families of square-integrable functions on $D$,
where $\mathbb I$ is at most countable.
Assume that $I:=\int_D a(x)a(x)^*d\mu(x)$ and $J:=\int_D b(x)b(x)^*d\mu(x)$ are positive definite, and that $J$
has finite trace and effective dimension
\[
\Tr(J) \,=\, \sum_{k\in \mathbb I} \|b_k\|_{L_2(D,\mu)}^2 \,<\,\infty
\quad\text{and}\quad M:=\frac{\Tr(J)}{\lambda_{\max}(J)}.
\]
Then, for any
$n\geq m$, there exist points $x_1,\dots,x_n\in D$ and weights ${w_1,\dots,w_n>0}$ such that
\[
\lambda_{\min}\left( \sum_{i=1}^{n} w_i\, 
a(x_i)a(x_i)^* \right) \geq
\left(1-\sqrt{\frac{m-1}{n}}\right)^2
\lambda_{\min}(I)
\]
and
\[
\lambda_{\max}\left(\sum_{i=1}^{n} w_i\, b(x_i)b(x_i)^* \right) \leq \left(1+\sqrt{\frac{M-1}{n}}\right)^2
\lambda_{\max}(J).
\]
\end{theorem}

\smallskip

One can retrieve Proposition~\ref{prop:BDM} from Theorem~\ref{thm:main} by taking $\mathbb I=\{1,\dots,M\}$ and assuming that $a$ and $b$ are orthonormal families. Indeed, in that case, we see that $I$ and $J$ are the $m\times m$ and $M\times M$ identity matrices. In turn, Proposition~\ref{prop:BDM} implies Proposition~\ref{thm:BSS} by taking $M=m$ and $b=a$.

\begin{remark}
The factor $m-1$ instead of $m$ in the lower bound is a minor improvement
already appearing in \cite{ChkifaDolbeault24}. It allows to treat the case $n=m$, with the bound
\[
1-\sqrt{\frac{m-1}{m}}>\frac{1}{2m}.
\]
In the same way, the factor $M-1$ instead of $M$ 
in the upper bound is most interesting when $J$ has rank one, where we recover \cite[Lemma~14]{BDM14}.
This case will be used in Corollary~\ref{coro:intro-equal} to obtain discretizations with equal weights. We will also leverage these improvements in Propositions~\ref{rk:weight_control} and Corollary~\ref{coro:noisy-recovery}, when considering families $a$ and $b$ containing one more function.
\end{remark}

\begin{remark}
Throughout the article, $L_2(D,\mu)$ and its subspaces are assumed to consist of complex-valued functions, not equivalence classes, which is needed for discretization. 
Moreover, it is possible to extend our results to Hilbert-valued functions, by combining them with \cite[Theorem 2.1]{BartelDung24}.
\end{remark}

The proof of Theorem~\ref{thm:main}, see Section~\ref{sec:proof}, 
builds upon the original proof of Proposition~\ref{thm:BSS},  
and the points and weights can be constructed in a quite similar way, see Algorithm~\ref{alg:abstract-construction} in Section~\ref{sec:implementation}.
Although the proof of Theorem~\ref{thm:main} is self-contained, 
we recommend reading the proof of \cite[Theorem~3.1]{BSS} first, since it comes with a very nice physical intuition.

Before we come to the proof, we first discuss several applications of this linear algebra result.
In Section~\ref{sec:discretization}, we discuss its implications to the problem of norm discretization.
In Section~\ref{sec:sampling}, we use these discretization results to obtain results on the problem of sampling recovery
and on the error of (weighted) least squares algorithms. The relation between these results is illustrated in Figure~\ref{fig:graph-results}.

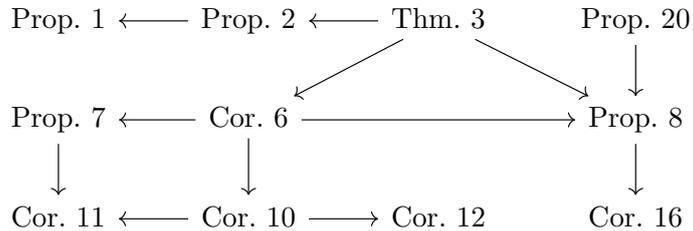
\begin{figure}[H]
\label{fig:graph-results}
\centering
\begin{tikzcd}
\rm Prop.~\ref{thm:BSS}
&\rm Prop.~\ref{prop:BDM} \arrow[l]
&\rm Thm.~\ref{thm:main} \arrow[l] \arrow[ld] \arrow[rd]
&\rm Prop.~\ref{prop:KW} \arrow[d] \\
\rm Prop.~\ref{rk:weight_control} \arrow[d]
&\rm Cor.~\ref{coro:main} \arrow[l] \arrow[d] \arrow[rr]
& &\rm Prop.~\ref{coro:intro-equal} \arrow[d] \\
\rm Cor.~\ref{coro:noisy-recovery}
&\rm Cor.~\ref{cor:recovery} \arrow[l] \arrow[r]
&\rm Cor.~\ref{cor:sampling-numbers-H}
&\rm Cor.~\ref{cor:recovery-unif}
\end{tikzcd}
\caption{Summary of the results of Sections~\ref{sec:intro}, \ref{sec:discretization} and \ref{sec:sampling}, with arrows denoting implications. The first line contains linear algebra results, the second deals with norm discretization and the third with sampling recovery. The lower left block details the reproducing kernel Hilbert space (RKHS) setting, with the second column directly applying Thm.~\ref{thm:main}, while the first column adds a control on the weights, and the third looks at sampling numbers. Lastly, the fourth column is focused on the $L_p$-setting, in which all weights are equal.}
\end{figure}

In fact, it is quite interesting that many of the recent advances in 
the areas of both norm discretization and sampling recovery
can be derived solely on the basis of Theorem~\ref{thm:main}.
Previous results used the non-constructive Kadison-Singer theorem from \cite{MSS},
see, e.g., \cite{CohenDolbeault,DKU,NSU},
which now becomes unnecessary.
This does not only simplify the proofs of those results
but also much improves the absolute constants.
Moreover, the least squares algorithms based on Theorem~\ref{thm:main} 
are implementable for various examples where 
an implementable construction of an optimal algorithm 
has previously been unknown
(e.g., for function spaces of mixed smoothness),
see Theorem~\ref{thm:constructive} and Example~\ref{ex:mixed-smoothness}.
A discussion of the practical implementation, along with numerical illustrations, is provided in Sections~\ref{sec:implementation} and~\ref{subsec:numerics}.

\section{Norm discretization}
\label{sec:discretization}

One of our two applications is the discretization of norms. This problem has recently received much attention. For further reading, we refer to the comprehensive treatments of the topic in \cite{Groe,KKLT,LMT,T18}. 
The BSS theorem (Proposition~\ref{thm:BSS}) can be stated 
equivalently as a result on $L_2$-norm discretization, 
where one aims to bound the $L_2$-norm of a function from a finite-dimensional space in terms of a discrete $\ell_2$-norm based on finitely many function evaluations.
Indeed, let $(D,\mu)$ be as in Proposition~\ref{thm:BSS}, and 
$V_m\subset L_2(D,\mu)$ with $\dim(V_m)=m$.  
Then, Proposition~\ref{thm:BSS} states that for every $n>m$, there exist points $x_1,\dots,x_n\in D$ and weights $w_1,\dots,w_n\in \R_+$ such that
\begin{equation} \label{eq:BSS-d}
\left(1-\sqrt{\frac{m}{n}}\right)\,\|f\|_2 
\;\leq\; 
\sqrt{\sum_{i=1}^n w_i|f(x_i)|^2}
\;\leq\; \left(1+\sqrt{\frac{m}{n}}\right)\,\|f\|_2
\end{equation}
for all $f\in V_m$, where we denote $\|\cdot\|_p=\|\cdot\|_{L_p(D,\mu)}$.
A similar result for general probability spaces $(D,\mu)$ is stated in \cite[Theorem~6.4]{DPSTT}.

Also our new result (Theorem~\ref{thm:main}) can be restated equivalently in terms of norm discretization.
In the upper bound,
the finite-dimensional space $V_m$
can be replaced with a reproducing kernel Hilbert space $H$, i.e.,
a Hilbert space of functions such that function evaluations are continuous (although the functions themselves are not necessarily continuous). 
We assume that $H$ is 
embedded 
in $L_2(D,\mu)$ continuously 
\begin{equation}\label{cond:continuous}
\lambda:=\sup_{f\in H} \frac{\Vert f \Vert_2^2}{\Vert f \Vert_H^2}<\infty
\end{equation}
and injectively
\begin{equation}\label{cond:injective}
 \Vert f \Vert_2 \ne 0
 \quad \text{for all} \quad f\in H \setminus \{0\},
\end{equation}
and that its kernel $K$ has finite trace 
\begin{equation}\label{eq:finite-trace}
 \Tr(K) := \int K(x,x)\,d\mu(x) < \infty.
\end{equation}
By \cite[Lemma~2.3]{Steinwart-Scovel}, this last condition ensures that the embedding $H\hookrightarrow L_2$ is Hilbert-Schmidt. In particular, \eqref{eq:finite-trace} implies \eqref{cond:continuous} with $\lambda\leq \Tr(K)$.

\smallskip

\begin{coro} \label{coro:main}
Let $(D,\mu)$ be a measure space, 
$V_m\subset L_2(D,\mu)$ be an $m$-dimensional space of functions and let $H$ be 
a reproducing kernel Hilbert space satisfying conditions \eqref{cond:continuous}, \eqref{cond:injective} and \eqref{eq:finite-trace}. 
For $n\geq m$, there exist points $x_1,\dots,x_n\in D$ and weights ${w_1,\dots,w_n>0}$ such that
\begin{equation}
\label{eq-low-coro-main}
\left(1-\sqrt{\frac{m-1}{n}}\right)\,\|f\|_2
\,\leq\,
\sqrt{\sum_{i=1}^n w_i|f(x_i)|^2}
\qquad \text{for all } f \in V_m
\end{equation}
and 
\begin{equation}
\label{eq-up-coro-main}
\sqrt{\sum_{i=1}^n w_i|g(x_i)|^2}
\,\leq\,
\left(1 +\sqrt{\frac{M-1}{n}} \right)\, \sqrt\lambda\,
\left\| g \right\|_H
\qquad \text{for all } g \in H,
\end{equation}
where $M=\Tr(K)/\lambda$ is the ``effective dimension'' of $H$ in $L_2(D,\mu)$.
\end{coro}

\begin{proof}
Let $a=(a_1,\hdots,a_m)^\top$ be an $L_2$-orthonormal basis of~$V_m$,
and $b=(b_1,b_2,\hdots)^\top$ be the basis of left singular vectors of the embedding $ H\hookrightarrow L_2$.
Then $I=\int a\,a^*d\mu$ is the identity, and
$J=\int b\,b^*d\mu$
is diagonal with trace $\Tr(J) =\Tr(K)$ and largest entry $\lambda_{\max}(J) = \lambda$,
since $b$ is orthonormal in $H$ and orthogonal in $L_2$.

Let $x_i$ and $w_i$ be the points and weights from Theorem~\ref{thm:main}, and denote
\[
G=\sum_{i=1}^{n} w_ia(x_i)a(x_i)^*\quad\text{and}\quad
\Gamma = \sum_{i=1}^{n} w_ib(x_i)b(x_i)^*.
\]
Then, any function $f\in V_m$ can be written as $f=\sum_{k=1}^{m} c_k a_k$ 
with $c\in \C^m$. Hence,
\[
 \sum_{i=1}^n w_i|f(x_i)|^2 = c\,G\,c^*
\geq \lambda_{\min}(G)\,|c|^2
 =\lambda_{\min}(G)\,\Vert f\Vert_2^2.
\]
In the same way, decomposing any $g\in H$ as $g=\sum_{k\in \mathbb I} c_k b_k$ with $c\in \ell_2(\II)$, we get\vspace{-4mm}
\[
 \sum_{i=1}^n w_i|g(x_i)|^2 = c\,\Gamma\,c^*
 \leq \lambda_{\max}(\Gamma)\,\Vert c \Vert_{\ell_2(\II)}^2
 =\lambda_{\max}(\Gamma)\,\Vert g \Vert_{H}^2.\qedhere
\]
\end{proof}

It is additionally possible to bound the sum of the weights as follows.
\begin{prop}
\label{rk:weight_control}
Let the assumptions of Corollary~\ref{coro:main} hold.
If $(D,\mu)$ is a probability space and $\int f\,d\mu=0$ for all $f\in H$,
then for $n\geq m$, there exist points $x_1,\dots,x_n\in D$ and weights ${w_1,\dots,w_n>0}$ 
such that \eqref{eq-low-coro-main} holds, \eqref{eq-up-coro-main} holds with $M$ instead of $M-1$, and 
\[
\sum_{i=1}^n w_i \leq \left(1+\sqrt\frac{M}{n}\right)^2.
\]
\end{prop}

\begin{proof}
Apply Corollary~\ref{coro:main} to $\widetilde H:=H\oplus \Span\{1\}$, equipped with the norm
\[
\|f+z\|_{\widetilde H}^2=\|f\|_{H}^2+\frac{|z|^2}\lambda,\qquad f\in H,\quad z\in \C.
\]
Indeed, $\widetilde H$ is also a reproducing kernel Hilbert space embedded in $L_2(D,\mu)$, with maximal eigenvalue
\[
\tilde \lambda:=\sup_{f \in \widetilde H} \frac{\Vert f \Vert_2^2}{\Vert f \Vert_{\widetilde H}^2} =\lambda,
\]
and its kernel $\widetilde K=K+\lambda$ has finite trace
$\Tr(\widetilde K)= \Tr(K)+ \lambda$.
This allows to translate \eqref{eq-low-coro-main} and \eqref{eq-up-coro-main} from $\widetilde H$ to $H$, and we additionally control the weights through
\[
\sqrt{\sum_{i=1}^n w_i}\leq \left(\sqrt{\lambda} +\sqrt{\frac{\Tr(K)}{n}} \right)\|1\|_{\widetilde H}=1+\sqrt{\frac{\Tr(K)}{n\lambda}}.\qedhere
\]
\end{proof}

The discretization inequality~\eqref{eq:BSS-d} is obtained for the special case where $H=V_m$, equipped with the $L_2$-norm, in which case 
$\lambda=1$ and $\Tr(K)=m$.
For the problem of sampling recovery,
as discussed in Section~\ref{sec:sampling},
we will consider infinite-dimensional spaces $H$.

Another interesting application are one-sided
discretization inequalities with equal weights. 
One-sided equal-weight discretization of the $L_2$-norm can be obtained in probability spaces $(D,\mu)$ by choosing $H$ as the one-dimen\-sional space of constant functions.
In combination with a result of Kiefer and Wolfowitz \cite{KW},
this can even be extended to the discretization of $L_p$-norms with $p>2$. 
Let $B(D)$ be the space of bounded functions on $D$.
In Section~\ref{sec:proof-coro}, we prove the following.

\begin{prop}
\label{coro:intro-equal}
Let $(D,\mu)$ be a probability space, and 
$V_m\subset L_2(D,\mu)\cap B(D)$ with $\dim(V_m)=m$. 
For $n\geq m$, there exist points $x_1,\hdots,x_{2n}\in D$ such that,
for all $f\in V_m$ and $2\leq p\leq\infty$, 
\[
\left(1-\sqrt{\frac{m}{n}}\right)
\cdot \|f\|_p 
\;\leq\; 
\,m^{\frac12-\frac1p}\cdot
\sqrt{\frac{1}{n}\sum_{i=1}^{2n} |f(x_i)|^2}.
\]
In the case $p\in\{2,\infty\}$, we can take $n$ points instead of $2n$ points.\\
In the case $n=m$, the factor $1-\sqrt{\frac{m}{n}}$ can be replaced by $\frac{1}{2m}$.
\end{prop}

We note that for $p=2$, the assumption $V_m\subset L_2$ is sufficient,
while for $p=\infty$, the assumption $V_m \subset B(D)$ is enough.
Up to constants, Corollary~\ref{coro:intro-equal} has been first proven in~\cite{BSU} in the case $p=2$, and in~\cite{KPUU25-uniform} for $p> 2$. 
There, the authors also rely on BSS~\cite{BSS}, 
but first generate a large enough random point set and then apply Proposition~\ref{thm:BSS}. 
With our direct approach, 
which is based on basically the same algorithm,
see Section~\ref{sec:implementation}, 
we can minimize oversampling and improve the constants.

In particular, by putting $n=2m$ and estimating the right hand side in terms of the maximum at the points, we obtain
that there are points $x_1,\hdots,x_{2m} \in D$ such that, for all $f\in V_m$,
\[
\|f\|_\infty 
\;\leq\; (2 + \sqrt{2})\, \sqrt{m} \cdot\,
\max_{1\leq i \leq 2m}\, \left|f(x_i)\right|.
\]

As the proof of Proposition~\ref{coro:intro-equal} refers to the proof of Theorem~\ref{thm:main}, 
we will present it after the proof of Theorem~\ref{thm:main} in Section~\ref{sec:proof}.

\begin{example}
For general spaces $V_m \subset L_2(D,\mu)$ with a probability space $(D,\mu)$, 
it is not possible to have a two-sided discretization with equal weights. 
Indeed, consider $D=\{0,1\}$ with $\mu(\{0\})=\eps\in (0,1)$, 
and let $V_2=\R^{\{0,1\}}$.
Assume that there are $n$ points $x_1,\dots,x_n\in D$ such that
\begin{equation}\label{eq:equal-2-sided}
\alpha \|f\|_2^2
\leq w \sum_{i=1}^n |f(x_i)|^2
\leq \beta \|f\|_2^2,\quad f\in V_2
\end{equation}
for some positive $\alpha$, $\beta$ and $w$.
The case $f=1$ shows that $\alpha \leq  wn$. 
Due to the lower bound, at least one of the points must be equal to 0, and the upper bound for the indicator function of $\{0\}$ gives the condition $w\leq \beta \eps$, hence $n \geq \frac{\alpha}{\beta} \eps^{-1}$.
As $\eps$ is arbitrarily small, 
there are no fixed values of $\alpha,\beta$ and $n$
such that \eqref{eq:equal-2-sided} can be obtained
for every 2-dimensional subspace of every $L_2$ probability space. 
\end{example}

\section{Sampling Recovery}
\label{sec:sampling}

Based on the discretization inequalities above, 
we easily obtain error bounds for corresponding 
weighted least-squares algorithms
on general approximation spaces~$V_m$.
Assume we are given points $x_1,\hdots,x_n \in D$
and weights $w_1,\hdots,w_n>0$.
The data is assumed to be of the form $y_i = f(x_i)$
for some unknown function $f$, which we want to learn.
We consider the weighted least squares approximation
\begin{equation}\label{def:wls}
\tilde f := \argmin_{g\in V_m} \sum_{i=1}^n w_i\, |y_i -g(x_i)|^2.
\end{equation}
In the following, the points and weights will always be chosen such that the minimizer is uniquely defined. 
As before, the points and weights can be constructed by Algorithm~\ref{alg:abstract-construction} that will be discussed in Section~\ref{sec:implementation}.
Via Corollary~\ref{coro:main},
we obtain the following.

\smallskip

\begin{coro}\label{cor:recovery}
Let $(D,\mu)$ be a measure space. 
Let $H$ be a reproducing kernel Hilbert space satisfying conditions \eqref{cond:continuous}, \eqref{cond:injective} and \eqref{eq:finite-trace}.
Let $V_m \subset H$ be an $m$-dimensional subspace and
let $n\geq m$.
Then, there are points $x_1,\hdots,x_n \in D$
and weights $w_1,\hdots,w_n>0$
such that, for any $f\in H$, the weighted least squares approximation~\eqref{def:wls} with $y_i=f(x_i)$ satisfies
\[
\big\|f- \tilde f\big\|_2 
\,\leq\, 
\left(1+\frac{1+s}{1-r}\right)\sqrt{\lambda_m} \,\big\|f-P_m f\big\|_H,
\]
where
\[
\lambda_m:=\sup_{f \in H_m} \frac{\Vert f \Vert_2^2}{\Vert f \Vert_{H}^2},
\qquad
r:=\sqrt\frac{m-1}{n}
\quad\text{and}\quad
s:=\sqrt{\frac{\Tr(K_m)-\lambda_m}{n\,\lambda_m}}.
\] 
Here, $H_m$ is the $H$-orthogonal complement of $V_m$,
$P_m$ is the $H$-orthogonal projection onto $V_m$,
and $K_m$ is the reproducing kernel of $(H_m,\|\cdot\|_H)$.
\end{coro}

\begin{proof}
Applying Corollary~\ref{coro:main} to the spaces $V_m$ and $H_m$ gives points and weights such that
\[
(1-r)\|\tilde g\|_2
\leq \sqrt{\sum_{i=1}^n w_i |\tilde g(x_i)|^2}
\leq \sqrt{\sum_{i=1}^n w_i |g(x_i)|^2}\leq (1+s)\sqrt{\lambda_m}\,\|g\|_{H},
\]
where $\tilde g\in V_m$ is the weighted least-squares projection of ${g=f-P_mf\in H_m}$.
As this projection is linear and sends $P_mf\in V_m$ to itself, $\tilde g=\tilde f-P_mf$. Hence
\[
\|f-\tilde f\|_2=\|g-\tilde g\|_2\leq \|g\|_2+\|\tilde g\|_2\leq \left(1+\frac{1+s}{1-r}\right)\sqrt{\lambda_m}\,\|g\|_{H}. \qedhere
\]
\end{proof}

In the case of noisy measurements $y_i=f(x_i)+e_i$ for arbitrary $e_i\in \C$, the following extension holds under additional assumptions.

\begin{coro}
\label{coro:noisy-recovery}
Assume that $(D,\mu)$ is a probability space, that $H$ satisfies conditions \eqref{cond:continuous}, \eqref{cond:injective} and \eqref{eq:finite-trace}, and that $V_m \subset H$ has dimension $m$. 
Moreover, assume that 
$V_m$ contains the Riesz representer of the integral, i.e., 
\[
h=\int K(x,\cdot)\,d\mu(x) \,\in V_m.
\]
Then, for $n \ge m$, 
there are points
$x_1,\hdots,x_n \in D$
and weights $w_1,\hdots,w_n>0$
such that, for any $f\in H$ and $e\in \C^n$, the least squares approximation \eqref{def:wls} on $V_m$ with noisy measurements $y_i=f(x_i)+e_i$ satisfies
\begin{equation}
\label{eq-prop-noisy}
\big\|f- \tilde f\big\|_2 
\,\leq\, 
\frac{1+\tilde s}{1-r}\, \Big(2\,\sqrt{\lambda_m} \,\big\|f-P_m f\big\|_H
+\,\Vert e \Vert_{\infty}\Big),
\end{equation}
with $\lambda_m$, $K_m$, $P_m$ and $r$ as in Corollary~\ref{cor:recovery}, 
and $\tilde s=\sqrt{\frac{\Tr(K_m)}{n\,\lambda_m}}$.
\end{coro}

\begin{proof}
Observe that $h\in H$ with $\|h\|_H\leq\sqrt{\Tr(K)}$. 
Moreover, for any function $f$ in $H_m$, the $H$-orthogonal complement of $V_m$, it holds
\[
0 = \<f,h\>_H = \int \<f,K(x,\cdot)\>_H\,d\mu(x) = \int f(x)\,d\mu(x).
\]
Replacing Corollary~\ref{coro:main} by Proposition~\ref{rk:weight_control} in the proof of Corollary~\ref{cor:recovery},
the weighted least squares projection $\tilde f$ of the noisy measurements $y_i=f(x_i)+e_i$ onto $V_m$ satisfies \eqref{eq-prop-noisy}.
Indeed, letting $\hat{f}$ be the solution for exact data $f(x_i)$, we bound 
$\Vert f - \hat{f}\Vert_2$ as in Corollary~\ref{coro:noisy-recovery}
(with $s$ replaced by $\tilde s$) and control the norm of 
$\hat{f} - \tilde f \in V_m$ via
\[
 (1-r)\Vert \hat{f} - \tilde f \Vert_2 
 \le \sqrt{\sum_{i=1}^n w_i |(\hat{f} - \tilde f)(x_i)|^2}
 \le \sqrt{\sum_{i=1}^n w_i |e_i|^2}
 \le (1+\tilde s) \Vert e \Vert_\infty,
\]
where the second inequality comes from the fact that
\[
\hat f - \tilde f = \argmin\limits_{g\in V_m} \sum_{i=1}^n w_i |e_i - g(x_i)|^2,
\]
which holds by linearity of \eqref{def:wls} as a function of $y \in \C^m$.
\end{proof}

\medskip

In some cases, the function $h$ in Corollary~\ref{coro:noisy-recovery} is naturally included in~$V_m$.
For example, if $D$ is a group and
both the kernel~$K$ and the measure~$\mu$ are 
translation-invariant, then~$h$ is a constant function. 
Without the assumption $h\in V_m$, we still have $h\in H$ and hence
we can simply use $\widetilde V_m := V_m\oplus \Span\{h\}$ as approximation space.
In this case, $\lambda_m$, $K_m$ and $P_m$ have to be used relative to $\widetilde V_m$, and $r$ is replaced by $\sqrt{m/n}$.

\medskip

We now state implications for the (linear) \emph{sampling numbers} of a class $F \subset Y$
in a (semi-)normed function space $Y \subset \C^D$, which are defined by 
\[
g_n^{\rm lin}(F,Y) 
\,:=\, 
\inf_{\substack{x_1,\dots,x_n\in D\\ \phi_1,\dots,\phi_n\in Y}}\, 
\sup_{f\in F}\, \Big\|f - \sum_{i=1}^n f(x_i) \, \phi_i\Big\|_{Y}.
\]
They measure the minimal worst-case error
of linear approximation algorithms that use at most $n$ samples.

Our result is directly applicable to sampling numbers in $Y=L_2$ of unit balls $F=B_H$ in a RKHS $H$.
It is obtained from Corollary~\ref{cor:recovery} on least squares approximation
by choosing the approximation space $V_m$ optimally.

\begin{coro}\label{cor:sampling-numbers-H}
Let $(D,\mu)$ be a measure space 
and $H$ a reproducing kernel Hilbert space 
satisfying \eqref{cond:continuous}, \eqref{cond:injective} and \eqref{eq:finite-trace}.
Then, for all $m\in \N$ and $n\geq m$,
\[
g_n(B_H,L_2) 
\,\leq\, \left(1+\frac{1}{1-r}\right)
\left(\sigma_m + \sqrt{\frac{1}{n}\sum_{k> m} \sigma_k^2} \right),
\]
where $r=\sqrt\frac{m-1}{n}$ and $\sigma_0 \geq \sigma_1 \geq \hdots$ are the singular values
of the embedding $H\hookrightarrow L_2$.
\end{coro}

\begin{proof}
  We use Corollary~\ref{cor:recovery}
  for the space $V_m$ spanned by the $m$ left singular functions having
  the $m$ largest singular values $\sigma_0,\hdots,\sigma_{m-1}$.
  In this case, 
  we have $\lambda_m=\sigma_m^2$ and $s=\sqrt{\frac1{n\,\sigma_m^2} \sum_{k> m} \sigma_k^2}$. For $f\in B_H$, as $\big\|f-P_m f\big\|_H\leq 1$, we obtain
  \[
    \big\|f- \tilde f\big\|_2 
    \,\leq\, \sigma_m+\frac{1}{1-r}\left(\sigma_m + \sqrt{\frac{1}{n}\sum_{k> m} \sigma_k^2} \right).
  \]
  Noting that the least squares approximation $\tilde f$ can be written as a linear sampling algorithm $\tilde f = \sum_{i=1}^n f(x_i) 
\phi_i$, the result is proved.
\end{proof}

The singular numbers of the embedding $H\hookrightarrow L_2$
are equal to the approximation numbers, 
Kolmogorov numbers, Gelfand numbers, and all other s-numbers.
In this sense, Corollary~\ref{cor:sampling-numbers-H} serves as a comparison 
of the power of linear sampling algorithms
with more general algorithms and other approximation benchmarks, see~\cite[Sec.\,10.1]{KU-acta} for a comprehensive treatment.

For convenience, we state the special cases $n=m$ and $n=2m$:
\begin{equation}\label{eq:n-is-m}
g_m(B_H,L_2) 
\,\leq\, (2 m+1) \left(\sigma_m + \sqrt{\frac{1}{m}\sum_{k> m}\sigma_k^2} \right).
\end{equation}
and 
\begin{equation}\label{eq:n-is-2m}
g_{2m}(B_H,L_2) 
\,\leq\, 5
\left(\sigma_m + \sqrt{\frac{1}{2m}\sum_{k> m} \sigma_k^2} \right).
\end{equation}

\smallskip

\begin{remark}[Comparison with earlier results]
Corollary~\ref{cor:sampling-numbers-H} improves on \cite[Theorem~23]{DKU}, 
by providing a constructive approach and reducing the constants.
Notably, \cite{DKU} requires an oversampling factor $n/m \geq 10^7$,
while we get the same bound already for $n/m=2$,
and reasonable results even for $n=m$.
Another improvement is that the underlying algorithm allows general spaces $V_m$. 
The corollary can also be applied to improve the bounds in \cite{KPUU}
on approximation in the $L_p$-norm, 
and answers the open problems from \cite[Remark~21]{DKU} and \cite[Remark~2]{KPUU}. 
It can also be seen as an improvement on~\cite{BSU}, 
which uses the constructive framework of~\cite{BSS} 
but still has a logarithmic factor in the error bound.
\end{remark}

Let us note that we can use Corollary~\ref{cor:sampling-numbers-H}
to improve the constants in various bounds on the sampling numbers
also for general non-Hilbert function classes $F$
from \cite{DKU,KPUU,KU2}.
We omit the details here, and refer to
Theorem~\ref{thm:constructive}
for a constructive version 
that does not require the Hilbert space structure.

\smallskip

\begin{remark}[Exponential decay]
The formula \eqref{eq:n-is-m} reveals another advantage over \cite[Theorem~23]{DKU}. 
Namely, if we have singular values with exponential or faster decay (e.g., for the Gaussian kernel), then the improvement is in a certain sense more than a constant since we can put $n=m$. 
For instance, if $\sigma_k=\alpha^k$ with $\alpha \in (0,1)$, 
\cite{DKU} gives an upper bound of the form $\beta^n$ for the $n$-th sampling number with some $\alpha<\beta<1$. 
The new bound~\eqref{eq:n-is-m} gives a bound of order $n\alpha^n$, which is better. See~\cite{KarvonenSuzuki}
for an example where this could be applied.
\end{remark}

\smallskip

\begin{remark}[RKHS assumption]
The above results are formulated for reproducing kernel Hilbert spaces~$H$ to save some technicalities, 
but they also hold for separable Hilbert spaces $H$ such that $H\hookrightarrow L_2$ is Hilbert-Schmidt. The latter assumption 
already implies that $H$ is \emph{almost} a RKHS in the following sense: 
If $H\hookrightarrow L_2$ is Hilbert-Schmidt, then every $f\in H$ has an a.e.-convergent Fourier series (w.r.t.~every $H$-orthonormal basis). 
Hence, if $H$ is separable, then there is a countable dense subset $H_0\subset H$ and $\mu$-null set $D_0\subset D$ 
such that every $f\in H_0$ restricted $D\setminus D_0$ has everywhere convergent series, 
and we can \emph{work} with such functions as if we had a RKHS with finite trace.
Since the set of acceptable points in our construction has positive measure, see Lemma~\ref{lemboth}, we can pick the sampling points in the ``good'' set $D \setminus D_0$.
By continuity of $H\hookrightarrow L_2$, the result then holds for all of~$H$.
See \cite[Sec.\,4]{MU} or \cite[Sec.\,2]{DKU} for details.
\end{remark}

\medskip

Finally,
we can use Proposition~\ref{coro:intro-equal}
on discretization with equal weights
to obtain error bounds for plain least squares approximation 
in terms of best uniform approximation.
This has been done before in many papers,
for instance \cite{BSU,ChkifaDolbeault24,KPUU25-uniform,PU,T20}.
Proposition~\ref{coro:intro-equal} can further improve upon
the corresponding results.
For the points $x_i \in D$ 
and data $y_i\in\C$, $1\leq i\leq 2n$,
we now consider the plain least squares approximation
\begin{equation}\label{def:pls}
\tilde f := \argmin_{g\in V_m} \sum_{i=1}^{2n} |y_i -g(x_i)|^2.
\end{equation}

\begin{coro}\label{cor:recovery-unif}
Let $(D,\mu)$ be a probability space, $V_m\subset L_2(D,\mu)\cap B(D)$ with $\dim(V_m)=m$, 
and $2\le p\le\infty$. 
For all $n> m$,
there are points $x_1,\hdots,x_{2n} \in D$
such that the plain least squares 
approximation~\eqref{def:pls} 
with $y_i = f(x_i) + e_i$
satisfies
\[
\big\|f- \tilde f\big\|_p 
\,\leq\, \left(1 + \sqrt{2}\cdot\frac{m^{1/2-1/p}}{1-\sqrt{m/n}}\right) \Big( \min_{g\in V_m} \,\big\|f-g\big\|_\infty + \Vert e \Vert_\infty \Big)
\]
for all $f\in B(D)$ and any noise $e_1,\hdots,e_n \in \C$.
For $p\in\{2,\infty\}$, we can use $n$ instead of $2n$ points, remove the factor $\sqrt{2}$, 
and it suffices to have $V_m \subset L_2$ or $V_m \subset B(D)$, respectively.
\end{coro}

We remark that, in general, the factor $m^{1/2-1/p}$ cannot be avoided,
see the discussion in \cite[Sec.\,5.3]{KPUU25-uniform}.

\begin{proof}
For all $g\in V_m$, the plain least-squares projection of $h=f-g$ from noisy measurements $h(x_i)+e_i$ is $\tilde h=\tilde f-g$.
Applying Proposition~\ref{coro:intro-equal} 
for the space $V_m$
and using the corresponding points $x_1,\hdots,x_{2n}$, we get
\begin{align*}
m^{\frac{1}{p}-\frac{1}{2}}\left(1-\sqrt{\frac{m}{n}}\right)\|\tilde h\|_p
&\leq \sqrt{\frac{1}{n}\sum_{i=1}^{2n} |\tilde h(x_i)|^2}\\
&\leq \sqrt{\frac{1}{n}\sum_{i=1}^{2n} |h(x_i)+e_i|^2}\leq \sqrt{2}\,(\|h\|_{\infty}+\|e\|_\infty).
\end{align*}
As $\mu$ is a probability measure, we also have $\|h\|_p\leq \|h\|_\infty$, which yields
\[
\|f-\tilde f\|_p = \|h-\tilde h\|_p \leq \|h\|_p+\|\tilde h\|_p
\leq \|h\|_\infty +\|\tilde h\|_p.
\]
It only remains to combine the two previous estimates and take the infimum over $g\in V_m$.
\end{proof}

In the same way we deduced Corollary~\ref{cor:sampling-numbers-H} from Corollary~\ref{cor:recovery}, we can use Corollary~\ref{cor:recovery-unif} to bound sampling numbers in $L_p$-spaces by Kolmogorov or Gelfand widths similar to \cite[Remark 1.5]{ChkifaDolbeault24} and \cite[Theorem 6]{KPUU}. We refer to \cite[Sec.\,3.2]{KU-acta} for a detailed discussion.

\section{Proofs of Theorem~\ref{thm:main} and Proposition~\ref{coro:intro-equal}}
\label{sec:proof}

\subsection{Recap on linear algebra in infinite dimension}

For convenience, before we start with the proof of the main result,
we collect a few properties for the computation with bounded operators on $\ell_2(\II)$, which we see as potentially infinite matrices $A = (A_{ij})_{i,j \in \II} \in \C^{\II \times \II}$.
We refer to \cite[Chapter VI]{reed} for a comprehensive treatment.
Consider the class of
positive definite invertible operators
\[
\mathcal P=\Big\{A\in\C^{\II \times \II}\colon A^*=A,\;\sup_{\|x\|_{\ell_2(\II)}=1} x^*Ax< \infty\quad\text{and}\quad \inf_{\|x\|_{\ell_2(\II)}=1} x^*Ax>0\Big\}.
\] 
In finite dimension, we denote by $\mathcal P_m$ the set of positive definite matrices in $\C^{m\times m}$, and recall that elements of $\mathcal P_m$ or $\mathcal P$ have an inverse in $\mathcal P_m$ or $\mathcal P$.

For self-adjoint operators $A$ and $B$, we write $A\preccurlyeq B$ (or $B\succcurlyeq A$) if $B-A$ is positive semi-definite
and $A\prec B$ (or $B\succ A$) if $B-A\in \mathcal P$. Clearly, $A\prec B$ implies $A\preccurlyeq B$.
Then, the following properties hold for bounded operators $A$ and $B$:
\begin{enumerate}[(i)]
    \item \label{property:pd} If $A$ is self-adjoint and $B\succcurlyeq 0$, then $ABA\succcurlyeq 0$.
    \item\label{property:trace-positive} If $A,B\succcurlyeq 0$ and $A$ has finite trace, then $AB$ has finite trace $\Tr(AB) \ge 0$.\\
    Moreover, if $A$ and $B$ are positive definite, $\Tr(AB)> 0$.
    \item  \label{property:SM} Shermann-Morrison formula: 
    If $A\in \mathcal P$, $a\in \ell_2(\mathbb I)$ and $w\in\R$ are such that $1+wa^*A^{-1}a > 0$,
    then $A+waa^*\in \mathcal P$ with
    \[
     (A+waa^*)^{-1} = A^{-1} - \frac{wA^{-1}aa^*A^{-1}}{1+wa^*A^{-1}a}.
    \]
\item\label{property:circ}
    If $a,b\in \ell_2(\II)$, then $\Tr(ab^*)=b^*a$.
    \item\label{property:CS} Cauchy-Schwarz inequality: If $A\succcurlyeq 0$ and $B\succcurlyeq 0$ have finite trace, then
    \[
    |\Tr(AB)|^2 \,\leq\, \Tr(A^2) \cdot \Tr(B^2)<\infty.
    \]
\end{enumerate}

\subsection{Proof of Theorem~\ref{thm:main}}
\label{sec:main-proof}

We first note that we may assume without loss of generality 
that $I$ is the identity matrix, as explained in \cite[Proof of Theorem~1.2]{BSS}.
Indeed, if $I$ is not the identity matrix,
we apply the theorem for $\tilde a := I^{-1/2} a $ in place of $a$.
Then $\widetilde I := \int \tilde a(x) \tilde a(x)^* d\mu(x)$ is the identity matrix.
We get points $x_i \in D$ and weights $w_i>0$ such that
\[
\sum_{i=1}^n w_i\tilde a(x_i) \tilde a(x_i)^*
\succcurlyeq
\left(1-\sqrt{\frac{m-1}{n}}\right)^2 \widetilde I.
\]
Multiplying from the left and right by $I^{1/2}$, 
we retrieve the statement by property~\eqref{property:pd}.

The basic idea of the proof is to select points $x$ and weights $w$ one by one,
in such a way that adding the rank-one updates $w a(x) a(x)^*$ and $wb(x)b(x)^*$ to matrices $A\in \mathcal P_m$ and $B\in \mathcal P$
modifies their eigenvalues in a controlled way.
This change is quantified via the \emph{lower} and \emph{upper potentials}
\[
\Phi(A):=\Tr(A^{-1})\quad\text{and}\quad \Psi(B):=\Tr(JB^{-1}).
\]
Note that $\Phi(A)$ and $\Psi(B)$, 
for $A\in \mathcal P_m$ and $B\in \mathcal P$,
are bounded and positive by property~\eqref{property:trace-positive} since 
the positive definite $J$ has finite trace,
\[
Y:=A^{-1}\in \mathcal P_m\quad\text{and}\quad W:=B^{-1}\in \mathcal P.
\]
In fact, we even have 
\[
A \succcurlyeq \Phi(A)^{-1} I
\quad\text{and}\quad
B \succcurlyeq \Psi(B)^{-1} J
\]
for all $A\in \mathcal P_m$ and $B\in \mathcal P$.
To see the second inequality, note that $B^{-1/2}JB^{-1/2} \succcurlyeq 0$ by property~\eqref{property:pd} and its norm is hence smaller than its trace, which is equal to $\Psi(B)$, and therefore $\Psi(B)I-B^{-1/2}JB^{-1/2} \succcurlyeq 0$;
multiplying with $B^{1/2}$ from both sides gives the inequality.

Before each step, we update $A$ and $B$ by \emph{increments} $\delta,\zeta>0$:
assuming that $A\succ \delta I$, we denote
\[
Z:=(A- \delta I)^{-1}\in \mathcal P_m \quad\text{and}\quad X:=(B+\zeta J)^{-1}\in \mathcal P,
\]
which satisfy $Z\succ Y$ and $X\prec W$. Taking traces, we obtain
\begin{equation}
\label{eq_comp_tr}
\Tr(Z)>\Tr(Y)=\Phi(A)
\quad\text{and}\quad
\Tr(JX)<\Tr(JW)=\Psi(B)
\end{equation}
by applying again property~\eqref{property:trace-positive}.
With this, we can define the \emph{lower verifier} 
\[
    L_A^{\delta}(a)
     \,:=\,\frac{a^*Z^2 a}{\Tr(Z-Y)}-a^*Za
\]
and \emph{upper verifier}
\[
U_B^{\zeta}(b)
    :=\frac{b^*X J Xb}{\Tr(JW-JX)}+b^*X b
     \ge 0, 
\]
which will be used to check 
whether the potentials of $A$ and $B$ increase when adding rank-one updates of the form $waa^*$ and $-w b  b^*$ and increments $-\delta I$ and $\zeta J$.
This is expressed in the following lemma. 

\smallskip

\begin{lemma}
\label{lemboth}
Let $A\in \mathcal P_m$, $B\in \mathcal P$,
$I\in\mathcal P_m$ be the identity, 
and $J\in\C^{\II\times\II}$ be positive definite with finite trace.
For $a\in \C^m$, $b\in\ell_2(\II)$ and $w>0$, 
and for $0<\delta<\Phi(A)^{-1}$ and $\zeta>0$,
consider the updated operators
\begin{equation*}
A'=A-\delta I+waa^*\quad\text{and}\quad B'=B+\zeta J-wbb^*.
\end{equation*}
\begin{enumerate}[a)]
    \item If $L_A^\delta(a) \ge \frac1w$,
    then $A' \in \mathcal P_m$ with $\Phi(A')\leq \Phi(A)$.
    \item If $U_B^\zeta(b) \le \frac1w$, then $B'\in \mathcal P$ with $\Psi(B')\leq \Psi(B)$.
\end{enumerate}
\end{lemma}

\begin{proof}
We follow \cite[Lemmas 3.3--3.4]{BSS}. 
We first note that
\[
\lambda_{\min}(A) \geq \Phi(A)^{-1} 
> \delta. 
\]
Hence $A\succ \delta I$ and $Z=(A-\delta I)^{-1}\in \mathcal P_m$.
Since $1+wa^*Za >0$ for any vector $a\in \C^m$ and weight $w>0$,
the Shermann-Morrison formula shows that $A'=Z^{-1}+waa^*\in \mathcal P_m$ with
\begin{equation}
\label{eq:lower-barrier-cond}
\Phi(A')
\overset{\eqref{property:SM}}{=}
\Tr\left(Z-\frac{wZaa^*Z}{1+w a^*Za}\right)\overset{\eqref{property:circ}}{=} \Tr(Z)-\frac{a^*Z^2a}{\frac{1}{w}+a^*Za}.
\end{equation}
Rearranging the terms, we see that $\Phi(A')\leq \Phi(A)$
if and only if the lower verifier satisfies $L_A^\delta(a) \geq \frac{1}{w}$.
 
We turn to the upper potential.
Since $B\in \mathcal P$, $X:=(B+\zeta J)^{-1}\in \mathcal P$.
Let $b\in \ell_2(\mathbb I)$ and $w>0$ such that $U_B^\zeta(b)\leq \frac{1}{w}$. Observing that either $Xb=0$ or
\[
\frac{1}{w}-b^*Xb\geq \frac{b^*X J Xb}{\Tr(JW-JX)}>0
\]
since $J$ is positive definite, we see that $1-wb^*Xb>0$ in any case.
The Sherman-Morrison formula then shows that the operator
$B'=X^{-1}-wbb^*$ belongs to $\mathcal P$, and yields
\[
 \Psi(B')
\overset{\eqref{property:SM}}{=}
\Tr(JX)+\frac{\Tr(JXbb^*X)}{\frac{1}{w}-b^*Xb}
\overset{\eqref{property:circ}}{=} \Tr(JX)+\frac{b^*XJXb}{\frac{1}{w}-b^*Xb}.
\]
Thus, if the upper verifier
satisfies $U_B^\zeta(b)\leq \frac{1}{w}$, we get $\Psi(B')\leq \Psi(B)$.
\end{proof}

We now turn to the $L_2$-framework.
In the rest of this section,
let $(D,\mu)$ be a measure space, and let $a=(a_1,\hdots,a_m)^\top$
and $b= (b_k)_{k\in \mathbb I}^\top$ be families of square-integrable functions,
where $\mathbb I$ is at most countable.
Further assume that 
$I:=\int_D a(x)a(x)^*d\mu(x)$ is the identity 
and that $J:=\int_D b(x)b(x)^*d\mu(x)$ is positive definite and has finite trace.

By Lemma~\ref{lemboth},
the condition $L_A^\delta(a(x)) \ge U_B^\zeta(b(x))$
implies the existence of $w>0$
such that the potentials of the updates
do not increase.
The following lemma shows when this condition can be satisfied.

\begin{lemma} \label{lemma:pos_measure}
For any $A\in \mathcal P_m$ and $0<\delta<\Phi(A)^{-1}$ and
for and any $B\in \mathcal P$ and $\zeta>0$,
it holds that
\[
\int L_A^\delta(a(x))\, d\mu(x) >  \frac{1}{\delta}-\Phi(A)
\quad\text{and}\quad
\int U_B^\zeta(b(x))\, d\mu(x) < \frac{1}{\zeta}+\Psi(B).
\]
In particular, if
\begin{equation}
\label{ineq-lemboth}
\frac{1}{\delta}-\Phi(A)\geq \frac{1}{\zeta}+\Psi(B),
\end{equation}
the set of points $x\in D$ with  
$L_A^\delta(a(x)) \ge U_B^\zeta(b(x))$ has positive measure
and
for any $x \in D$ and $w>0$ with $U_B^\zeta(b(x)) \le \frac1w \le L_A^\delta(a(x))$, the updates
\begin{equation}\label{eq:updates}
A'=A-\delta I+wa(x)a(x)^*\quad\text{and}\quad B'=B+\zeta J-wb(x)b(x)^*
\end{equation}
again satisfy~\eqref{ineq-lemboth}.
\end{lemma}

\begin{proof}
We follow \cite[Lemma 3.5]{BSS}.
Cleary, the `In particular'-part follows from the first part and Lemma~\ref{lemboth}.
Recalling that $Y:=A^{-1}\in \mathcal P_m$ 
and $Z:=(A- \delta I)^{-1}\in \mathcal P_m$ from the previous proof, we have
\[
Y^{-1}=Z^{-1}+\delta I.
\]
Multiplying this equality by $Y$ from the left and $Z$ and $Z^2$ from the right
yields
\[
Z=Y+\delta YZ\quad\text{and}\quad Z^2=YZ+\delta YZ^2.
\]
By a Cauchy-Schwarz inequality~\eqref{property:CS} between $Y^{1/2}$ and $Y^{1/2}Z$,
we have
\[
\Tr(Y)\Tr(Z^2)-\Tr(Z)\Tr(YZ)=\delta\Big(\Tr(Y)\Tr(YZ^2)- |\Tr(YZ)|^2\Big) \geq 0.
\]
Combining this with the equality $\Tr(Z-Y)=\delta\Tr(YZ)$, we compute the average of $L_A^\delta$:
\begin{equation}\label{eq:exp-LA}
\int_D L_A^\delta(a(x))\,d\mu(x)
=\frac{\Tr(Z^2)}{\delta\Tr(YZ)}-\Tr(Z)
\geq \frac{\Tr(Z)}{\Tr(Y)}\left(\frac{1}{\delta}-\Tr(Y)\right),
\end{equation}
which is larger than $\frac{1}{\delta}-\Phi(A)$ by \eqref{eq_comp_tr}.

On the other hand, for the upper potential, using $W:=B^{-1}\in \mathcal P$
and $X:=(B+\zeta J)^{-1}\in \mathcal P$, we get
\[
X^{-1}=W^{-1}+\zeta J\quad\text{and}\quad W=X+\zeta WJX,
\]
and hence $\Tr(JW-JX)=\zeta \Tr(JWJX)$. As $JWJ\succcurlyeq 0$ and $XJX\succcurlyeq 0$ by property~\eqref{property:pd}, we also have
\[
 \Tr(JWJX) - \Tr(JXJX)
 = \zeta  \Tr(JWJ \cdot XJX) \geq 0
\]
by property~\eqref{property:trace-positive}.
As a result,
\begin{equation}\label{eq:exp-UB}
\int_D U_B^\zeta(b(x))\,d\mu(x)=\frac{\Tr(JXJX)}{\Tr(JW-JX)}+\Tr(JX)
\leq \frac{1}{\zeta}+\Tr(JX),
\end{equation}
which is smaller than $\frac{1}{\zeta}+\Psi(B)$ by \eqref{eq_comp_tr}.
\end{proof}

We refer to Section~\ref{subsec:Christoffel} for a lower bound on the 
probability of the event $L_A^\delta(a(x))\geq U_B^\zeta(b(x))$ when we sample $x$ w.r.t.~the \emph{Christoffel density}.

\medskip

\begin{proof}[Proof of Theorem~\ref{thm:main}]
First assume that $m>1$ and $M:=\frac{\Tr(J)}{\lambda_{\max}(J)}\geq 1+\frac{1}{n}$, and denote
\[
r=\sqrt\frac{m-1}{n}
\quad\text{and}\quad
s=\sqrt{\frac{M-1}{n}}.
\]
We fix the increments
\[
\delta=\frac{1-r}{n}\quad\text{and}\quad \zeta = \frac{1+s}{n},
\]
and proceed by induction with 
\begin{equation}\label{eq:initialize}
A_0=\frac{\delta m}{r}\,I \in \mathcal P_m
\quad\text{and}\quad
B_0=\frac{\zeta\cdot\Tr(J)}{s}\,I \in \mathcal P,
\end{equation}
where we denote by $I$ the identity of appropriate dimension.
We compute
\[
\frac{1}{\delta}-\Phi(A_0)=\frac{1-r}{\delta}=n
=\frac{1+s}{\zeta}=
\frac{1}{\zeta}+\Psi(B_0).
\]

Applying Lemmas~\ref{lemboth} and~\ref{lemma:pos_measure} $n$ times,
we obtain points $x_1,\hdots,x_n \in D$
and weights $w_1,\hdots,w_n >0$
such that for $a(x_i)\in\C^m$ and $b(x_i)\in \ell_2(\II)$, 
the operators 
\[
A_n:=A_0-n\delta I+\sum_{i=1}^{n} w_i a(x_i) a(x_i)^*
\]
and
\[
B_n:=B_0+n\zeta J-\sum_{i=1}^{n} w_i b(x_i) b(x_i)^*
\]
belong to $\mathcal P_m$ and $\mathcal P$, respectively, and their potentials are bounded by
\[
\Phi(A_n)\leq \Phi(A_0)=\frac{r}{\delta}\quad\text{and}\quad
\Psi(B_n)\leq \Psi(B_0)=\frac{s}{\zeta}.
\]
As
$A_n \succcurlyeq \Phi(A_n)^{-1} I \succcurlyeq \frac{\delta}{r}\, I$, this gives
\begin{equation}\label{eq:LB}
\sum_{i=1}^{n} w_i a(x_i) a(x_i)^*
\succcurlyeq \Big(n\delta+\frac{\delta}{r}-\frac{m\delta}{r}\,\Big)I = (1-r)^2I.
\end{equation}
In the same way, as $B_n \succcurlyeq  \Psi(B_n)^{-1}J \succcurlyeq \frac{\zeta}{s}\, J$, we use the assumption $s\geq\frac{1}{n}$ and the bound $J \preccurlyeq \lambda_{\max}(J) I$ to obtain
\begin{equation}\label{eq:UB}
\sum_{i=1}^{n} w_i b(x_i) b(x_i)^*
\preccurlyeq \Big(n\zeta -\frac{\zeta}{s}\Big)J+\frac{\zeta}{s}\Tr(J)\, I \preccurlyeq (1+s)^2\lambda_{\max}(J) \, I,
\end{equation}
which is the desired result.

In the case $m=1$, we have $r=0$, hence $A_0$ is not defined and the above proof does not apply. However, we can replace the lower verifier $L_A^\delta(a(x))$ by $L(a(x))=n|a(x)|^2$, 
where we identify $a=(a_1)$ and $a_1$,
since it holds
\[
\int_D L(a(x))\,d\mu(x)=n.
\]
Hence, $\int L(a(x)) \ge \int U_{B_{i-1}}(b(x))$ for $i=1$,
and inductively choosing $x_i$ and $w_i$ with $L(a(x_i)) \ge \frac{1}{w_i} \ge U_{B_i}^\zeta(b(x_i))$,
also for $i=2,...,n$.
This results in points and weights that satisfy the upper bound~\eqref{eq:UB} and $w_i\geq 1/L(a(x_i))$.  Noting that the $a(x_i)$ are scalar,
\[
\sum_{i=1}^n w_i |a(x_i)|^2\geq \sum_{i=1}^n \frac{1}{n}=\lambda_{\min}(I).
\]
Similarly, if $M<1+\frac{1}{n}$, 
we replace the upper verifier $U_B^\zeta(b(x))$ by
\[
U(b(x))=\frac{n}{\Tr(J)}\,\sum_{k\in\mathbb{I}}|b_k(x)|^2,
\]
which satisfies
\[
\int_D U(b(x))\,d\mu(x)=n.
\]
Hence, $\int L^\delta_{A_{i-1}}(a(x)) \ge \int U(b(x))$ for $i=1$,
and inductively choosing $x_i$ and $w_i$ with $L^\delta_{A_{i-1}}(a(x_i)) \ge \frac{1}{w_i} \ge U(b(x_i))$,
also for $i=2,...,n$.
This results in points and weights that satisfy the lower bound~\eqref{eq:LB} and $w_i\leq 1/U(b(x_i))$.
We conclude with
\[\begin{split}
\lambda_{\max}\left(\sum_{i=1}^{n} w_i b(x_i) b(x_i)^*
\right)
&\leq\sum_{i=1}^{n} w_i \sum_{k\in\mathbb{I}}|b_k(x_i)|^2 
\leq
\Tr(J)=M\lambda_{\max}(J) \\ 
&\leq (1+s)^2\lambda_{\max}(J),
\end{split}\]
where for the last inequality $s<\frac{1}{n}$ implies $M=1+ns^2\leq 1+s$.
\end{proof}

\begin{remark}
Unlike \cite{BSS}, we decided to include the increments $\delta I$ and $\zeta J$ in $A$ and $B$, instead of defining upper and lower barriers. Apart from that, the potential $\Phi$ is the same as the one used in \cite{BSS}. The main difference appears in the upper potential $\Psi$, due to the multiplication by $J$. In particular, note that the initialization $B_0$ is a multiple of the identity, while the increments with $\zeta$ are multiples of $J$.
\end{remark}

\subsection{Proof of Proposition~\ref{coro:intro-equal}}
\label{sec:proof-coro}

The proof of Proposition~\ref{coro:intro-equal} is as in \cite[Thm.\,13]{KPUU25-uniform},
but instead of \cite[Thm.\,6.2]{BSU}, we use a special case of Corollary~\ref{coro:main}.
The case $p=\infty$ is treated with 
the following result of Kiefer and Wolfowitz~\cite{KW};
see also \cite[Prop.\,9]{KPUU25-uniform} for additional details.

\begin{prop}[\cite{KW}]\label{prop:KW}
Let $D$ be a set and $V_m$ be an $m$-dimen\-sion\-al subspace of~$B(D)$. 
For every $\varepsilon>0$, 
there exists a finitely supported measure $\mu_\eps$ on~$D$,
such that 
\[
 \|f\|_{\infty} \,\leq\, \sqrt{m+\varepsilon} \cdot \Vert f \Vert_{L_2(D,\mu_\eps)}. 
\]  
\end{prop}

\medskip

\begin{proof}[Proof of Proposition~\ref{coro:intro-equal}]
For $p=2$, we simply apply Corollary~\ref{coro:main} with $H$ the space of constant functions, that is, with $b(x)=1$ and $J=\begin{pmatrix}1\end{pmatrix}$. By the 
proof of Theorem~\ref{thm:main} (for the case $M<1+\frac1n$), it is possible to take weights
\[
w_i=\frac{1}{U(b(x_i))}=\frac{\lambda_{\max}(J)}{n\,|b(x_i)|^2}=\frac{1}{n}.
\]
The lower bound then yields the desired result since for all $n\geq m$,
\[
1-\sqrt{\frac{m-1}{n}}> \max\left(1-\sqrt{\frac{m}{n}},\frac{1}{2m}\right).
\]

For the case $p=\infty$, take $\varepsilon>0$ such that the LHS in the previous inequality is larger than $\sqrt{1+\varepsilon/m}$ times the RHS. The result follows by applying again Corollary~\ref{coro:main} with $H=\Span\{1\}$, but with $\mu$ replaced by the measure $\mu_\varepsilon$ from Proposition~\ref{prop:KW}.

Lastly, for $2<p<\infty$, let $x_1,\hdots,x_n$ be the the points from the case $p=2$
and $x_{n+1},\hdots,x_{2n}$ be the points from the case $p=\infty$.
Since $\mu$ is a probability measure,
\[
\Vert  f \Vert _{L_p(D,\mu)}  \leq 
 \Vert  f  \Vert_{L_2(D,\mu)}^{\frac{2}{p}} \cdot  \Vert  f \Vert_{\infty}^{1-\frac{2}{p}},
\]
and we conclude by combining the two previous bounds.
\end{proof}

\section{Construction of points and weights}
\label{sec:implementation}

The proof of Theorem~\ref{thm:main} not only implies that the set of admissible samples $(x_1,\dots,x_n)$ has positive measure, it also provides a criterion for finding these points. This is summarized in Algorithm~\ref{alg:abstract-construction} below, where we reuse the notations from Section~\ref{sec:main-proof}. However, two issues remain:
we did not specify how to suggest points $x_i$ on line~\ref{state:oracle};
and the basis $b$ could be infinite-dimensional, preventing the computation of $U_B^\zeta(b(x_i))$ and $B'$ on lines~\ref{state:test} and~\ref{state:update}.

\begin{algorithm}
\begin{algorithmic}[1]
\smallskip

\STATE{{\bf Input:} $m\in \mathbb N$, $n\geq m$, $a\in L_2(D,\mu)^m$, $b\in L_2(D,\mu)^{\mathbb I}$
}
\smallskip
\STATE{{\bf Set up:} $A\gets A_0$, $B\gets B_0$ as in \eqref{eq:initialize}}
\smallskip
\FOR{$\displaystyle i=1,\dots,n$}
\STATE{\textbf{Suggest} a sampling point $x_i \in D$
\label{state:oracle}}
\IF{ $L_A^\delta(a(x_i))\geq U_B^\zeta(b(x_i))$
\label{state:test}}
\STATE choose $w_i>0$ such that $1/w_i$ is between these values \label{state:choose-w}
\ELSE 
\STATE go to line~\ref{state:oracle}
\ENDIF
\STATE {\textbf{Update} $A \gets A'$, $B \gets B'$ as in \eqref{eq:updates}\label{state:update}}
\ENDFOR
\STATE{{\bf Output:} Points $(x_i)_{1\leq i\leq n}$ and weights $(w_i)_{1\leq i\leq n}$}
\end{algorithmic}
\caption{Construction of the points and weights from Theorem~\ref{thm:main}}
\label{alg:abstract-construction}
\end{algorithm}

Regarding the first issue, one option is to find a finite subset of $D$ in which at least one point satisfies the condition of line~\ref{state:test}.
This 
candidate set
could be a rank-1 lattice \cite{BKPU}, an admissible mesh \cite{BPV}, or a large initial sample \cite{CDdomains}.
Constructing points by searching within a larger
finite set is often referred to as \emph{subsampling},
and has been performed in \cite{BKPU,BSU}
for algorithms based on~\cite{BSS} in the case $b=a$.
Note that we do not have to know a priori whether the candidate set is suitable, we can simply test our favorite set.

Another way of dealing with this issue is to propose a point $x_i$ at random.
Similarly to~\cite{ChkifaDolbeault24}, 
we show in Section~\ref{subsec:Christoffel} below that a random sample from the so-called \emph{Christoffel density}
will pass the criterion from line~\ref{state:test}, and hence will be accepted, with a probability of order $1/m$.

The second issue is addressed in Section~\ref{subsec:truncation}, where we analyse the error from truncating the infinite basis $b$ at some index $N$. In practice, we do not restrict ourselves to the RKHS setting, and consider functions $f$ in the span of an orthonormal system $(\eta_k)_{k\in\II}$ in 
$L_2(D,\mu)$ such that, for any $k\in \II$ and $x\in D$, the point value $\eta_k(x)$ can be computed in time $O(1)$.
Given a function $f\in L_2(D,\mu)$ and a finite set $I\subset \II$, we define 
the best approximation of $f$ in $\Span\{\eta_k,\; k\in I\}$ as
\[
 f_I := \sum_{k\in I} \hat{f}(k)\, \eta_k,
 \quad\text{where}\quad
 \hat{f}(k) := \int_D f \,\overline{\eta_k}\, d\mu.
\]
In order to bound the truncation error $f-f_I$ in uniform norm, we assume that the functions $\eta_k$ are bounded in $L_\infty$, with a bound of the form
\begin{equation}\label{eq:assum-uniform-bound}
    \sum_{k\in I_\ell} |\eta_k(x)|^2 \,\leq\, C_{\eta}^2\, \ell^{2\theta},\qquad x\in D,\quad \ell\in\N
\end{equation}
for some finite $\theta \geq 1/2$ and $C_\eta>0$, and where $(I_\ell)_{\ell\in \N}$ is a nested sequence of subsets of $\II$ with $|I_\ell|\leq \ell$. Since the integral of the left hand side in~\eqref{eq:assum-uniform-bound} equals $|I_\ell|$,
the value $\theta=1/2$ is the best possible value (unless $|I_\ell| = o(\ell)$).
Examples where this best possible behavior occurs are given by the Fourier, Walsh, Haar, or Chebyshev system with their corresponding orthogonality measure $\mu$.
However, note that the value of $\theta$ will not appear in our error bound;
it will appear in the result by restricting the class of functions for which the error bound is applicable.
Also, condition \eqref{eq:assum-uniform-bound} can be avoided if another way of bounding $\Vert f - f_I \Vert_\infty$ is available, see Remark~\ref{rem:uniform-decay}.

Combining the above condition on $(\eta_k)$ with a decay assumption on the coefficients of $f$, we prove that the truncation error can be controlled for a set $I_N$ of large cardinality $N$. As this integer can be much larger than $m$ and~$n$, we present in Section~\ref{subsec:fast-comp} an implementation whose cost is polynomial in $n$ but linear in $N$. We note that, as the algorithm produces deterministic guarantees from randomly proposed points, its computational complexity is necessarily a random variable.
In the end, we obtain our constructive result.

\begin{theorem}
\label{thm:constructive}
Let $(D,\mu)$ be a probability space, $(\eta_k)_{k\in\II}$ an $L_2$-orthonormal system and $(I_\ell)_{\ell\in\N}$ nested index sets with $|I_\ell|\le\ell$ that satisfy condition~\eqref{eq:assum-uniform-bound}, $m\geq 2$, $n=2 m$, $\alpha_0>\theta$ and $N=\lceil m^{\alpha_0/(\alpha_0-\theta)}\rceil$.
Assume that the constant function $1$ belongs to $\{\eta_k,\;k\in I_m\}$. Algorithm~\ref{alg:abstract-construction} can be implemented, with $a$ and $b$ defined in \eqref{eq:a-b-constructive}, in time
\[
O\big((c+Nn)n^2\mathcal G\big),
\]
where $c$ is the cost of drawing one point from the Christoffel density, and ${\mathcal G\in\N}$ is a shifted geometric random variable of parameter $1/2$. It yields points $x_1,\dots,x_n\in D$ and weights $w_1,\dots,w_n>0$ which satisfy the following.
If $f\in L_2(D,\mu)$ with
\[
\|f-f_{I_\ell}\|_2\leq C_f\, \ell^{-\alpha}\log^\beta(\ell)
\]
for some $\alpha\geq \alpha_0$ and $\beta\geq 0$, and if
$f(x_i)=\lim_{\ell\to\infty}f_{I_\ell}(x_i)$ for $1\leq i\leq n$,
the weighted least-squares approximation \eqref{def:wls}
from these values satisfies
\[
\|f-\tilde f\|_2\leq C_f\,C\, m^{-\alpha}\log^\beta(m),
\]
where $C$ depends on $\theta$, $\alpha_0$, $\alpha$, $\beta$ and $C_\eta$. 
\end{theorem}

In summary, given \eqref{eq:assum-uniform-bound},
we can construct $2m$ points and weights in (randomized) polynomial time
such that the corresponding weighted least squares estimator has the same rate of convergence as best approximation from the $m$-dimensional spaces, provided that the rate of best approximation is large enough.
The rest of this section is dedicated to the proof of Theorem~\ref{thm:constructive}.

\subsection{Christoffel sampling}
\label{subsec:Christoffel}

One possible implementation of line~\ref{state:oracle} in the algorithm consists of drawing candidate points for $x_{i}$ according to the Christoffel density associated to $V_m$,
\[
d\rho(x) = \frac{1}{m}\sum_{k=1}^m|a_k(x)|^2\,d\mu(x).
\]
For some $V_m$, this density can be given explicitly, and even sampled from. 
However, by assumption \eqref{eq:assum-uniform-bound} of Theorem~\ref{thm:constructive}, 
sampling from $\rho$ 
can be done via rejection sampling from the underlying measure with an acceptance probability of order at least $m^{-(2\theta-1)}$.
See e.g.~\cite[Sec.\,5]{CDdomains} and \cite{HA25} for alternative strategies 
to generate samples from the Christoffel density.

Recall that a suggested point $x$ will be accepted if $L_A^\delta(a(x)) \ge U_B^\zeta(b(x))$.
By using a slightly smaller increment $\delta$ for the lower barrier in the proof of Theorem~\ref{thm:main}, a good acceptance probability is achieved in exchange for a slightly worse lower frame bound.

\begin{lemma}
\label{lem:accept_proba}
If $x$ is drawn according to the Christoffel measure $\rho$ and the increment $\delta=\frac{1-r}{n}$ is replaced by
\[
\delta_\epsilon=\frac{1-r-\epsilon}{n}
\]
for some parameter $\epsilon < (1-r)^2$ in the computation of $L_A^\delta$ and the updates $A\leftarrow A'$ (but not in the initialization $A_0$),
then the acceptance probability satisfies
\[
\mathbb P\left(L_A^{\delta_\eps}(a(x)) \ge U_B^\zeta(b(x))\right) \ge \frac{\epsilon}{m}.
\]
In addition, any points $x_1,\dots x_n$ returned by our algorithm for the modified value of $\delta$ satisfy the upper bound \eqref{eq:UB} while the lower bound \eqref{eq:LB} becomes
\[
\lambda_{\min}\left(\sum_{i=1}^{n} w_ia(x_i)a(x_i)^*\right)
\geq  \Big(n\delta_\epsilon+\frac{\delta}{r}-\frac{m\delta}{r}\,\Big)= (1-r)^2-\epsilon.
\]
\end{lemma}
In the sequel, we will take $\epsilon=\frac{(1-r)^2}{2}$, which changes \eqref{eq:LB} by a factor 2.
\begin{proof}
We start with $A \in \mathcal P_m$ and $B\in \mathcal P$ such that \eqref{ineq-lemboth} holds for the original $\delta=\frac{1-r}{n}$.
Let
\[
\Delta(x):= \frac{L_A^{\delta_\epsilon}(a(x))- U_B^\zeta(b(x))}{\sum_{k=1}^m|a_k(x)|^2}.
\]
We now consider a random variable $x\sim \rho$.
According to Lemma~\ref{lemma:pos_measure},
the expectation of $\Delta(x)$ is 
\begin{align*}
\E(\Delta(x))
=\int_D \frac{L_A^{\delta_\epsilon}(a(x))-U_B^\zeta(b(x))}m \, d\mu(x)
\geq \frac{1}{m}\left(\frac{1}{\delta_\epsilon}-\Phi(A)- \frac{1}{\zeta}-\Psi(B)\right).
\end{align*}
Thanks to condition \eqref{ineq-lemboth}, 
we get
\[
\E(\Delta(x))
\geq \frac{1}{m}\left(\frac{1}{\delta_\epsilon}-\frac{1}{\delta}\right)=\frac{\epsilon}{m \,\delta_\epsilon(1-r)}.
\]
On the other hand, $\Delta(x)$ is almost surely bounded by
\begin{align*}
\Delta(x)
\leq \sup_{\substack{a\in\C^m 
}} \frac{L_A^{\delta_\eps}(a)}{\sum_{k=1}^m|a_k|^2} 
\leq \frac{\lambda_{\max}(Z^2)}{\Tr(Z)-\Tr(Y)}\leq  \frac{\Tr(Z^2)}{\delta_\epsilon\Tr(YZ)}\leq \frac{1}{\delta_\epsilon(1-r)},
\end{align*}
where $Y$ and $Z$ are from Section~\ref{sec:main-proof} with $\delta_\eps$ in place of $\delta$. Here,
the third inequality follows from $Z-Y=\delta_\epsilon YZ$
and the last inequality comes from the fact that
\[
\Tr(Z^2)-\Tr(YZ)=\delta_\epsilon \Tr(Y Z^2)\leq \delta_\epsilon\Tr(Y)\Tr(Z^2)\leq r\Tr(Z^2),
\]
by a Cauchy-Schwarz inequality~\eqref{property:CS} between $Y^{1/2}$ and $Y^{1/2}Z$ and
since $\Tr(Y)=\Phi(A)\leq r/\delta_\epsilon$. In the end, we obtain an acceptance probability
\[
\mathbb P(\Delta(x)\geq 0)
\geq \frac{\int_D \Delta\, d\rho}{\sup_{x\in D} \Delta(x)}\geq \frac{\epsilon}m.
\]
Note that, by Lemma~\ref{lemboth}, choosing $x_i$ with $\Delta(x_i) \ge 0$ and $w_i$ accordingly, we obtain that the potentials do not increase 
when switching to the updates 
$A':=A-\delta_\eps I+wa(x_i)a(x_i)^*$ and 
$ B':=B+\zeta J-wb(x_i)b(x_i)^*$ and
so 
condition~\eqref{ineq-lemboth} remains true with the original $\delta$ also for the updated operators.
\end{proof}

In total, when $n=O(m)$, we therefore need $\mathcal O(m^2)$ candidate points drawn according to the Christoffel functions, in order to implement the algorithm. More precisely, the number of draws is a geometric random variable of expectation $\mathcal O(m^2)$.
We cannot hope to use less Christoffel points, as illustrated by the following example.

\begin{example}
\label{ex:christoffel}
For each $m\in\N$, consider $D=[0,m]$ with the Lebesgue measure $\mu$
and define a RKHS $H$ on $D$ by its kernel $K(x,y)=\sum_{k\in \II} \sigma_k^2 b_k(x) b_k(y)$, where the functions
\[
b_k := \left\{\begin{array}{cl}
\mathbbm 1_{[k-1,k]}  &\text{ for } 1\leq k\leq m,\\
\cos(2\pi (k-m)x) \mathbbm 1_{[0,1]} &\text{ for } k>m,\\
\sin(2\pi (m-k)x) \mathbbm 1_{[0,1]} &\text{ for } k<-m,\\
\end{array}\right.
\]
are orthonormal in $L_2(D)$,
and $\sigma_k:=\frac1{|k|}$ for $k\in\II:=\Z\setminus\{-m,\dots,0\}$. Let also $V_m=\Span\{b_k \colon 1\leq k\leq m\}$. From~\eqref{eq:n-is-2m}, we get
\[
 g_{2m}^{\rm lin}(B_H,L_2) \,\leq\, \frac{10}m.
\]
On the other hand,
let $A_n$ be an approximation that uses information from at
most $n$ samples $x_1,\hdots,x_n$.
Let $\{y_1,\hdots,y_{n'}\} = \{x_1,\hdots,x_n\} \cap [0,1]$
be the samples in $[0,1]$. Observing that $A_n(-f)=A_n(f)$ if $f(x_1)=\dots f(x_n)=0$,
\[
 \sup_{\Vert f \Vert_H \leq 1} \Vert f - A_n(f) \Vert_{L_2([0,m])}
 \geq \sup_{\substack{\Vert f \Vert_H \leq 1 \\ f(x_i)=0}} \Vert f \Vert_{L_2([0,m])}
 \geq \sup_{\substack{\Vert f \Vert_H \leq 1 \\ f(y_i)=0}} \Vert f \Vert_{L_2([0,1])}.
 \]
As $\|f\|_{L_2([0,1])}\geq \big|\int_0^1f\big|$, the lower bound from \cite[Corollary~2]{KV} with $\lambda_0=1$ and $\lambda_k=\sigma_{|k|+m}^2$ for $k\neq 0$ gives
 \[
\sup_{\substack{\Vert f \Vert_H \leq 1 \\ f(y_i)=0}} \Vert f \Vert_{L_2([0,1])}
 \geq \sqrt{\min\bigg\{\frac{\lambda_0}2, \ \frac{1}{8n'} \sum_{k> 4n'} \lambda_k\bigg\}}
 \geq \frac{1}{8\max(n',\sqrt{mn'})}.
\]
As a consequence, the approximation $A_n$ can only attain an optimal error if $n' \gtrsim m$.
As the Christoffel density of $V_m$ is uniform on $[0,m]$,
if $x_1,\hdots,x_n$ are drawn i.i.d. from that density, 
we must have $n \gtrsim m^2$ with high probability. In summary, any 
algorithm achieving an error of optimal order 
must use $\Omega(m^2)$ Christoffel samples.
\end{example}

\begin{remark}
We remark that sampling from the 
adaptive quasi-density $a(x)^*Z^2a(x)$
would instead lead to a constant acceptance probability. However, we do not see an efficient way for sampling from it 
directly, see also~\cite{ChkifaDolbeault24}.
A different sampling density has been introduced in \cite{KU1},
namely,
\[
 d\tilde\varrho(x) = \frac{1}{2} d\rho(x) + \frac{1}{2} \frac{K_m(x,x)}{\Tr(K_m)} d\mu(x),
\]
where $K_m$ is the kernel of $H_m$, the orthogonal complement of $V_m$.
With this density (and weights $w_i= 1/(n\tilde \rho(x_i))$), 
it is proven that
$n=\mathcal O(m\log m)$ i.i.d. samples suffice with high probability.
In the particular example above,
the density results in choosing approximately $1/2$ of the points in $[0,1]$
instead of a ratio $1/m$.
\end{remark}

\begin{remark}
    In a different setting,
    where we consider instead the expected error $\mathbb E\,\Vert f -\tilde f \Vert_2^2$
    of random reconstructions,
    it is known that already $\mathcal{O}(m\log m)$ i.i.d. Christoffel samples
    give optimal results, see~\cite{CM}.
    In that setting, better constants can also be attained using determinantal point processes \cite{BBC}.
    In particular, in the interpolation regime $n=m$, the factor $m$ in \eqref{eq:n-is-m} can be replaced by $\sqrt{m}$ when considering randomized sampling numbers.
\end{remark}

We add that it would be 
possible to keep the value $\delta=(1-r)/n$ in Lemma~\ref{lem:accept_proba}, with a smaller acceptance probability.
Taking $w_i=1/L_A^\delta(a(x_i))$ for the weights, we get $\Phi(A')=\Phi(A)$ by~\eqref{eq:lower-barrier-cond}, thus the lower potential remains equal to its initial value 
\[
\Phi(A)=\Tr(Y)=\frac{r}{\delta}\quad\text{and}\quad\frac{1}{\delta}-\Tr(Y)=n.
\]
Recall that we also have $\frac{1}{\zeta}-\Psi(B)\leq n$, and observe that
\[
\Tr(Z)-\Tr(Y)=\delta\Tr(YZ)
\overset{\eqref{property:trace-positive}}{\geq}
\delta\Tr(Y^2)
\overset{\eqref{property:CS} }{\geq}
\frac{\delta}{m}\big(\Tr Y\big)^2
=\frac{r}{m}\Tr(Y).
\]
Then, according to \eqref{eq:exp-LA} and \eqref{eq:exp-UB}, 
the expectation of $\Delta(x_i)$ w.r.t. $x_i\sim d\rho$ is
\begin{align*}
\E(\Delta(x_i))
=\int_D \frac{L_A^\delta(a(x))-U_B^\zeta(b(x))}m \, d\mu(x)
\geq \frac{\Tr(Z)}{\Tr(Y)}\,\frac nm - \frac nm
\geq \frac{rn}{m^2}.
\end{align*}
Combining this calculation with the uniform bound on $\Delta$, the acceptance probability is
\[
\mathbb P(\Delta(x)\geq 0)\geq \frac{\int_D \Delta d\rho}{\sup_{x\in D} \Delta(x)}\geq \frac{rn\delta(1-r)}{m^2}=\frac{r(1-r)^2}{m^2},
\]
which is smaller than in Lemma~\ref{lem:accept_proba} by a factor $m$.

\begin{remark}
Example~\ref{ex:christoffel} suggests that the inequality $\lambda_{\max}(Z^2)\leq \Tr(Z^2)$ in the proof of Lemma~\ref{lem:accept_proba} is sharp, while the above Cauchy-Schwarz inequality $\Tr(Y^2)\geq \frac 1 m (\Tr Y)^2$ is not. 
In other words, the matrix $A$ only has a few small eigenvalues, which explains why the potential $\Phi(A)$ gives a tight control on $\lambda_{\min}(A)$ in the BSS approach.
\end{remark}

\subsection{Constructive results via truncation}
\label{subsec:truncation}

For the sake of simplicity, in this subsection, we assume that $\II=\N$ and that $I_\ell=\{1,\dots,\ell\}$, which only amounts to reordering the indices, and can therefore be done without loss of generality. We also denote $f_m=f_{I_m}$.

Additionally, all $\log$ functions are meant in base $2$, which does not influence the statement of Theorem~\ref{thm:constructive}.

\begin{proof}[Proof of the error bound of Theorem~\ref{thm:constructive}]
We first construct the RKHS $H=\Span\{\eta_1,\dots,\eta_N\}$ with kernel
\[
 K(x,y) := \sum_{k=1}^N k^{-2t} \eta_k(x) \overline{\eta_k(y)},\quad\text{where}\quad t=\frac{\alpha_0+\theta}{2}.
\]
This allows to apply Corollary~\ref{coro:noisy-recovery} to the function $f_{N}\in H$ and the space $V_m=\Span\{\eta_1,\dots,\eta_m\}$. More precisely, we take
\begin{equation}
\label{eq:a-b-constructive}
a=(\eta_k)_{1\leq k\leq m}\quad\text{and}\quad b=(k^{-t}\eta_k)_{m< k\leq N}.
\end{equation}
Then $\lambda_m=(m+1)^{-2t}$, $P_mf_{N}=f_{m}$ and
\[
\Tr(K_m)=\sum_{m<k\leq N} k^{-2t}\leq \lambda_m+\int_{m+1}^\infty \frac{dx}{x^{2t}}\leq \lambda_m\frac{m+2t}{2t-1},
\]
which is finite since $t>\theta\geq 1/2$. With the choice $n=2m$, we get
\[
\frac{1}{1-r}\leq 4\quad\text{and}\quad 1+\tilde s\leq 2\sqrt{\frac{2t}{2t-1}}\leq 2\sqrt{\frac{2\alpha_0}{\alpha_0-\theta}}.
\]
Therefore, Corollary~\ref{coro:noisy-recovery}, with measurements $y_i=f(x_i)=f_N(x_i)+e_i$, and the modification of the increment $\delta$ from Lemma~\ref{lem:accept_proba}, gives
\begin{equation} \label{eq:truncation}
\|f_{N}-\tilde f\|_2\leq 16\sqrt{\frac{2\alpha_0}{\alpha_0-\theta}}\,\Big(2\,m^{-t} \,\big\|f_N-f_m \big\|_H
+\,\Vert f-f_N \Vert_{\infty}\Big).
\end{equation}
We then control the two terms on the RHS by dyadic decompositions. We have
\begin{align*}
\|f_N-f_m\|_H^2
&\leq \sum_{k>m}k^{2t}\,|\hat f(k)|^2\\[-1em]
&= \sum_{\ell\geq 0}
\;\sum_{k=2^\ell m+1}^{2^{\ell+1}m}\;
k^{2t}\,|\hat f(k)|^2\\
&\leq \sum_{\ell\geq 0} \,(2^{\ell+1}m)^{2t}\,\Vert f - f_{2^\ell m} \Vert_2^2\\
&\leq C_f^2\,C_H^2\, m^{2t-2\alpha}\,\log^{2\beta}(m),
\end{align*}
where
$
C_H:= 2^{t+\beta}\Big(1+\sum_{\ell\geq 1}2^{\ell(2t-2\alpha)}\ell^{2\beta}\Big)^{1/2}
$
is finite since $\alpha>t$.

We control $\|f-f_N\|_\infty$ in the same way.
For all $x\in D$,
\begin{align*}
    \vert f(x) - f_{N} (x) \vert
    \,&\leq\, 
    \sum_{\ell\geq 0} \;\sum_{k = 2^\ell N+1}^{2^{\ell+1}N} \;\vert\hat f(k)\vert \, \vert \eta_k(x) \vert\\
    \,&\leq\,
    \sum_{\ell\geq 0} \|f-f_{2^\ell N}\|_2\, \Big(\sum_{k \leq 2^{\ell+1}N} \vert \eta_k(x) \vert^2\Big)^{1/2}\\
    \,&\leq\,
    C_f\,C_\infty N^{\theta-\alpha} \, \log^\beta N,
\end{align*}
where
$
C_\infty:=C_\eta 2^\theta\Big(1+\sum_{\ell\geq 1}2^{\ell(\theta-\alpha)}\ell^{\beta}\Big)
$
is finite since $\alpha>\theta$. By the choice of $N$, we have
\[
N^{\theta-\alpha}\leq m^{-\alpha}\quad\text{and}\quad \log^\beta(N)\leq 
\left(\frac{2\alpha_0}{\alpha_0-\theta}\right)^\beta 
\log^\beta(m),
\]
and therefore
\[
\|f-\tilde f\|_2\leq \|f-f_N\|_2+\|f_N-\tilde f\|_2\leq  C_f\,C\,m^{-\alpha}\log^\beta(m),
\]
where
\[
C=(1+32 C_H+16C_\infty)\left(\frac{2\alpha_0}{\alpha_0-\theta}\right)^{\beta+1/2}.\qedhere
\]
\end{proof}

\begin{remark}
\label{rem:uniform-decay}
    As can be seen from
    the proof, see~\eqref{eq:truncation}, the uniform bound~\eqref{eq:assum-uniform-bound} on the basis can be avoided and it may be possible to choose a smaller truncation index $N$ if there is other knowledge about $\Vert f - f_{N} \Vert_\infty$,
    possibly via smoothness assumptions.
    For example, if we fix parameters $t>1/2$, $p>1$ and truncate at $N = \lceil m^p \rceil$, a bound $\Vert f - f_m \Vert_2 \in O(m^{-\alpha})$ with $\alpha >t$ will result in a corresponding bound $\|f-\tilde f\|_2 \in O(m^{-\alpha})$ under the additional assumption that $\Vert f - f_{N} \Vert_\infty \in O(N^{-\alpha/p})$.
\end{remark}

 \begin{remark}
 Given a function $f$, the decay of the coefficients of $f$ ensures that its expansion in the basis $(\eta_k)$ converges almost everywhere. As the distribution of $(x_i)$ is absolutely continuous with respect to $\mu^{\otimes n}$, the assumption $f(x_i)=\lim_{\ell\to\infty}f_{I_\ell}(x_i)$ is almost surely true for each $f$.
 \end{remark}

\begin{example}
\label{ex:mixed-smoothness}
    An example, where Theorem~\ref{thm:constructive} could be applied
    is the approximation on Sobolev spaces $H^\alpha_{\rm mix}$
    of $d$-variate periodic functions with mixed smoothness $\alpha>1/2$.
    These can be defined by 
    \[
     H^\alpha_{\rm mix} =
     \left\{ f \in C([0,1]^d) \colon \Vert f \Vert_{\alpha}^2 := \sum_{ k \in \Z^d} w_{ k}^{2\alpha} |\hat{f}( k)|^2 < \infty \right\},
    \]
    where $\hat{f}( k)$ are the usual Fourier coefficients and $w_{ k} = \prod_{i=1}^d \max\{1,|k_i|\}$. 
    For integer $\alpha$, they can also be defined as spaces of functions possessing (weak) derivatives 
    in $L_2$ up to mixed order $\alpha$,
    see, e.g., \cite[Sec.\,2.1]{KSU} for a few details,
    and \cite{DTU} for a comprehensive treatment of dominating mixed smoothness.
    If we choose the functions $\eta_k$ as the trigonometric monomials
    with the frequencies $k \in \Z^d$ ordered according to their hyperbolic distance $w_{k}$ to the origin
    and $\alpha_0>1/2$,
    the algorithm from Theorem~\ref{thm:constructive} satisfies
    \begin{equation}\label{eq:order-mix}
     \Vert f - \tilde f \Vert_2 \,\lesssim\, m^{-\alpha} 
     (\log m)^{\alpha(d-1)} 
     \cdot \Vert f \Vert_{\alpha},
    \end{equation}
    for any smoothness $\alpha\geq \alpha_0$.
    
    We present this example since
    for the class $H^\alpha_{\rm mix}$,
    no (polynomial time) construction was previously known to give the optimal order~\eqref{eq:order-mix}.
    Only the existence of such algorithms was known due to~\cite{DKU}.
    Previous constructions that come close to the optimal order include sparse grids (Smolyak's algorithm), see~\cite{SU},
    random points~\cite{KU1},
    subsamples of random points~\cite{BSU}
    and generated sets (lattices with a non-integer generating vector), see~\cite{CNW}.
    \end{example}

\subsection{Fast computations}
\label{subsec:fast-comp}

If implemented naively, Algorithm~\ref{alg:abstract-construction} would have a cost $O(N^3)$ per iteration due to the inversion of $N\times N$ matrices when computing $U_B^\zeta(b(x_i))$ on line~\ref{state:test}. This cost becomes prohibitive for large values of $N$, which occur in Theorem~\ref{thm:constructive} when $\alpha$ is close to $\theta$, or when $p$ is large in Remark~\ref{rem:uniform-decay}.

However, it can be exploited that the matrices to be inverted after $i$ iterations are of the form
\[
W=(D-MM^*)^{-1}\quad\text{and}\quad X=(D'-MM^*)^{-1},
\]
where $D=B_0+i\zeta J$ and $D'=D+\zeta J$ are diagonal by orthogonality of the basis $b$, and
\[
M = [\sqrt{w_1} b_1|\dots|\sqrt{w_i} b_i] \in \C^{N\times i}
\]
is thin and tall since $i\leq n\ll N$ for the choice $n=2m$. By the Woodbury matrix identity, we get
\[
W=D^{-1}+D^{-1}M(I-M^*D^{-1}M)^{-1}M^*D^{-1},
\]
and similarly for $X$. The matrix $I-M^*D^{-1}M$ is now of more reasonable size $i\times i$, it can be assembled in time $O(Nn^2)$ and inverted in time $O(n^3)$.

In the end, we never explicitly construct matrix $B$. The computation of
\[
\Tr(JW)=\Tr(J D^{-1})+\Tr((I-M^*D^{-1}M)^{-1}\cdot M^*D^{-1}JD^{-1}M),
\]
as well as that of $\Tr(JX)$, takes $O(Nn^2)$ since $M^*D^{-1}JD^{-1}M$ can also be assembled in time $O(Nn^2)$. Moreover, once $I-M^*{D'}^{-1}M$ is known, evaluating $Xb$ for some vector $b\in \C^N$ takes $O(Nn)$ by successive matrix-vector products.
With this strategy, we are ready finish the proof of Theorem~\ref{thm:constructive}.

\begin{proof}[Proof of the computational complexity in Theorem~\ref{thm:constructive}]
At each iteration, we compute $Y$ and $Z$ in time $O(n^3)$, and the inverses of $I-M^*D^{-1}M$ and $I-M^*{D'}^{-1}M$ in $O(Nn^2)$, which allows to access $\Tr(JW-JX)$ at the same cost.

We then sample a point $x_i$ from the Christoffel measure $\rho$ at a cost $c$, evaluate $L_A^\delta(a(x_i))$ in $O(n^2)$ and $Xb(x_i)$ in $O(Nn)$, from which we evaluate $b(x_i)^*Xb(x_i)$, $(Xb(x_i))^*JXb(x_i)$, and hence of $U_B^\zeta(b(x_i))$, in $O(N)$.

In conclusion, recalling that $n\ll N$, the cost of all matrix inversions and trace estimations over the $n$ iterations is $O(Nn^3)$. Proposing and testing samples costs $c+O(Nn)$, and the total number of proposed points is described by a negative binomial distribution with $n$ successes and a success probability $\Omega(1/n)$ by Lemma~\ref{lem:accept_proba}. We conclude by bounding this random number in law by $O(n^2)\mathcal G$, where $\mathcal G\geq 1$ is a shifted geometric variable, satisfying $\mathbb P(\mathcal G=k)=2^{-k}$ for $k\in\N$.
\end{proof}

\section{Numerical results}
\label{subsec:numerics}

We now discuss some numerical examples
for the sampling points and weights 
from Theorem~\ref{thm:constructive}.
Figures~\ref{fig:trig-blocked}--\ref{fig:legendre-new} show the sampling points in $D=[0,1]^2$ 
for different choices of basis functions and weights.
In all cases, the points in the algorithm are suggested randomly and uniformly distributed, 
and the weight in line~\ref{state:choose-w} is chosen minimal. 
In addition, in each step, 
all previous points are tested again before new points are proposed. 
If one of the previous points is accepted, we only have to increase its weight,
thus reducing the size of the final point set.

For a given univariate orthonormal system $(\eta_k)_{k\in\N_0}$, 
we consider the tensor product bases 
\[
a=\bbrackets{\eta_j\otimes \eta_k\co\; j,k \in\N_0 
\;\text{ with } \;\sigma_{j,k} \leq R }
\]
and 
\[
b=\bbrackets{\frac{1}{\sigma_{j,k}}\cdot \eta_j\otimes \eta_k\co\; j,k\in \N_0
\;\text{ with } \; R<\sigma_{j,k} \leq R' 
\;\text{ or }\; (j,k)=(0,0)}.
\] 
Here, the weights describe either
\begin{itemize}
    \item isotropic smoothness, i.e., $\sigma_{j,k}=\max\{1,2\pi(j^2+k^2)\}$, or 
    \item mixed smoothness, i.e., $\sigma_{j,k}=\max\{1,2\pi|j|\}\cdot\max\{1,2\pi|k|\}$.
\end{itemize}
As for univariate basis functions, we consider 
\begin{itemize}
    \item \emph{trigonometric polynomials}: $\eta_k^T(x):=e^{2\pi{\rm i}k x}$, or
    \item \emph{Legendre polynomials}: 
    $\displaystyle \eta_k^L(x) := \sqrt{2k+1}\cdot \frac{(-1)^k}{k!} \frac{d^k}{dx^k} \left[\, x^k (1 - x)^k \,\right]$, or
    \item \emph{Chebyshev polynomials}: $\eta_0^C=1$ and 
    \[
    \eta^C_k(x) := \sqrt{2}\,\cos(k \arccos(2x-1)), \qquad k\ge1.
    \]
\end{itemize}
The first two are orthonormal 
w.r.t.~the Lebesgue measure, 
while the latter is orthonormal w.r.t.~the \emph{arcsine density} $\nu(x)=(\pi \sqrt{x(1-x)})^{-1}$.

For the trigonometric system, condition~\eqref{eq:assum-uniform-bound} is satisfied with $\theta = 1/2$. 
In this case, the Christoffel density gives the uniform distribution. 
Also for the Chebyshev system, condition~\eqref{eq:assum-uniform-bound} is satisfied with $\theta = 1/2$. 
For the Legendre system, we can choose $\theta=1$.
If we choose the truncation parameter $R'$ large enough 
(as can be computed via Theorem~\ref{thm:constructive}), 
for any previously chosen threshold $\alpha_0>\theta$,
the resulting least squares approximation hence gives 
optimal $L_2$-approximation rates
if the rate of best $L_2$-approximation in this basis
(ordered according to $\sigma_{j,k}$)
is at least $\alpha_0$.

We add that the true acceptance rate of the suggested points was a lot higher than the theoretical lower bounds from Section~\ref{subsec:Christoffel}; 
usually every second to tenth randomly generated point was accepted by the algorithm. 
For this reason, we also did not have to decrease the increment $\delta$ to $\delta_\eps$
as done in Section~\ref{subsec:Christoffel}
to increase the theoretical acceptance rate;
we used precisely the parameters as given in Algorithm~\ref{alg:abstract-construction}.

In all pictures, the size of the points is drawn proportional to the size of the corresponding weight. 
In addition to the points and weights,
we also plot the area of all points that would be rejected by the algorithm
in the following iteration.
That is, the blue area can be interpreted as the area 
that is `blocked' by the already chosen sampling points.
Two observations that we find interesting are:
\begin{itemize}
    \item The blocked area is roughly ball-shaped in the isotropic cases
    and roughly hyperbolically shaped in the hyperbolic case, see Figure~\ref{fig:trig-blocked}.
    This was not clear a priori, since the corresponding shape is a priori given only on the Fourier side,
    not on the sampling domain.
    Towards the end of the construction,
    the area of acceptable points takes roughly this shape, see Figure~\ref{fig:trig-final}.
    This fits to the knowledge that, in the isotropic case,
    the error is dominated by the size of the largest empty ball
    (see, e.g., \cite{KS1}).
    In the hyperbolic case, there is no theoretical result undermining this observation.
    
\begin{minipage}{.95\linewidth}
\begin{figure}[H]
\centering
\includegraphics[clip, width=.5\linewidth, trim=30mm 20mm 20mm 0mm]{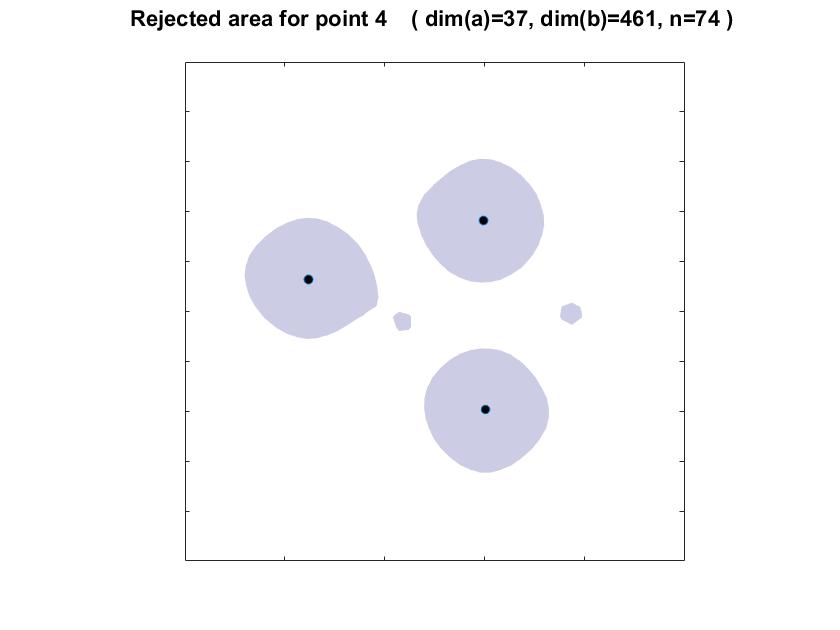}
\hspace{-5mm} 
\includegraphics[clip, width=.5\linewidth, trim=30mm 20mm 20mm 0mm]{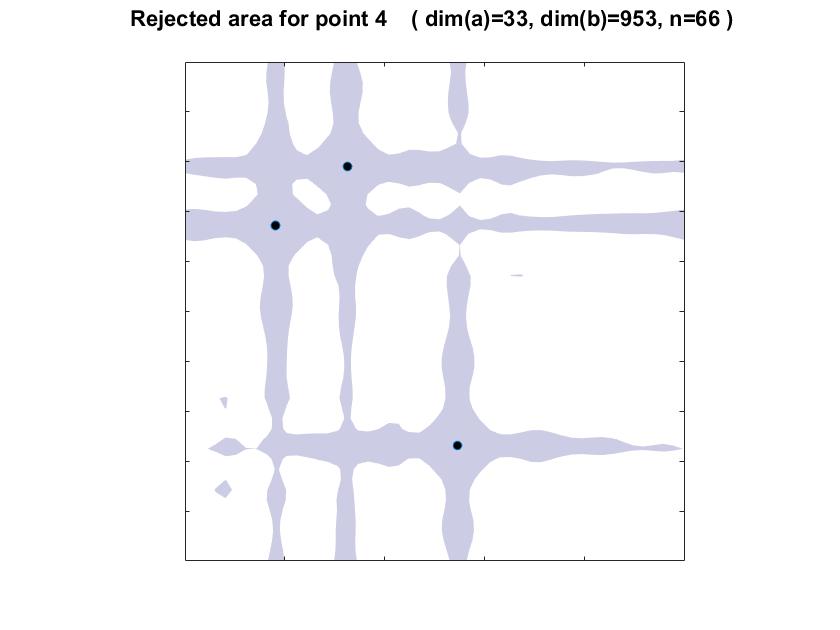} 
\caption{Points rejected by the algorithm (grey area) 
for the Fourier basis with isotropic (left)  and mixed (right) smoothness.}
\label{fig:trig-blocked}
\end{figure}
\end{minipage}

\begin{minipage}{.95\linewidth}
\begin{figure}[H]
\centering 

\includegraphics[clip, width=.5\linewidth, trim=30mm 20mm 20mm 0mm]{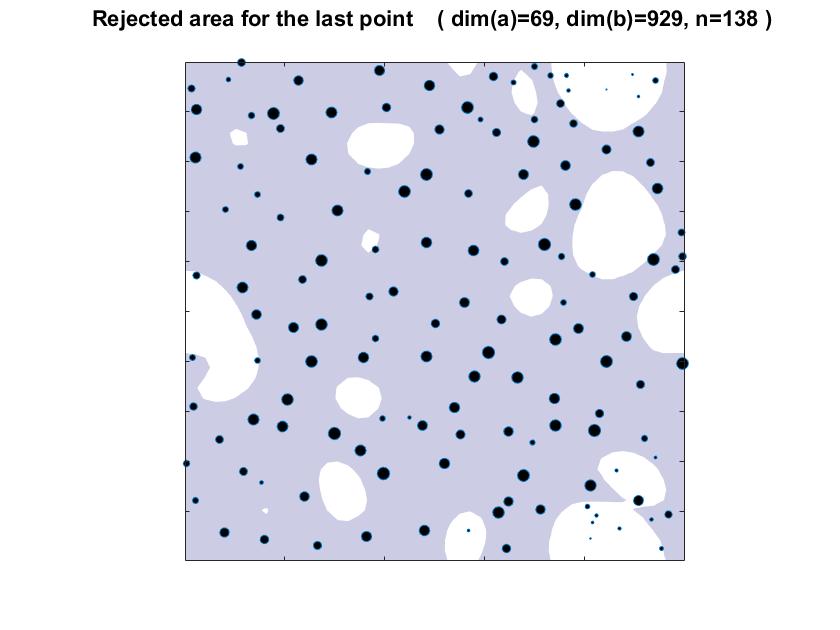}
\hspace{-3mm} 
\includegraphics[clip, width=.5\linewidth, trim=30mm 20mm 20mm 0mm]{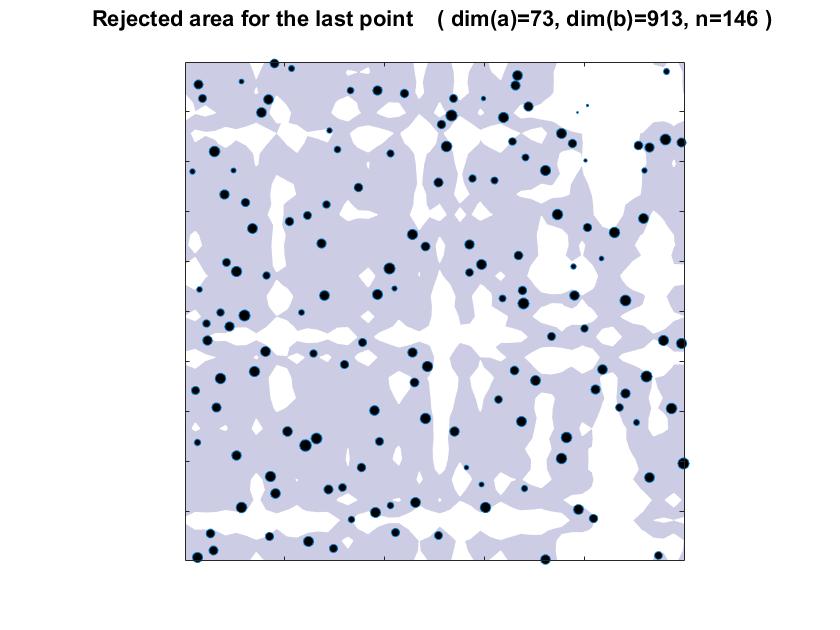} 
\hspace{-5mm}
\caption{Points generated according to the algorithm for the Fourier basis with isotropic (left)  and mixed (right) smoothness, with the area of points that would be rejected next (grey).}
\label{fig:trig-final}
\end{figure}
\medskip

\end{minipage}
    
    \item As one would expect and hope,
    for the trigonometric system, the points are evenly spread over the domain.
    The points for the Legendre and Chebyshev systems have a much higher concentration towards the boundary. The weights of the points at the boundary, however, are chosen larger for Chebychev polynomials, but smaller for Legendre polynomials, 
    see Figures~\ref{fig:cheb-new} and~\ref{fig:legendre-new}.

\begin{minipage}{.95\linewidth}
\begin{figure}[H]
\includegraphics[clip, width=.5\linewidth, trim=30mm 20mm 20mm 0mm]{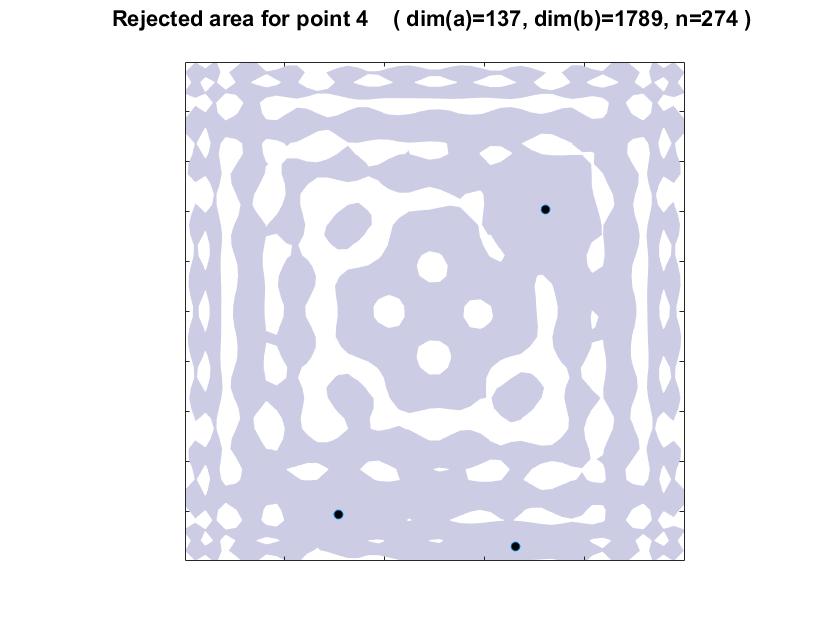} 
\includegraphics[clip, width=.5\linewidth, trim=30mm 20mm 20mm 0mm]{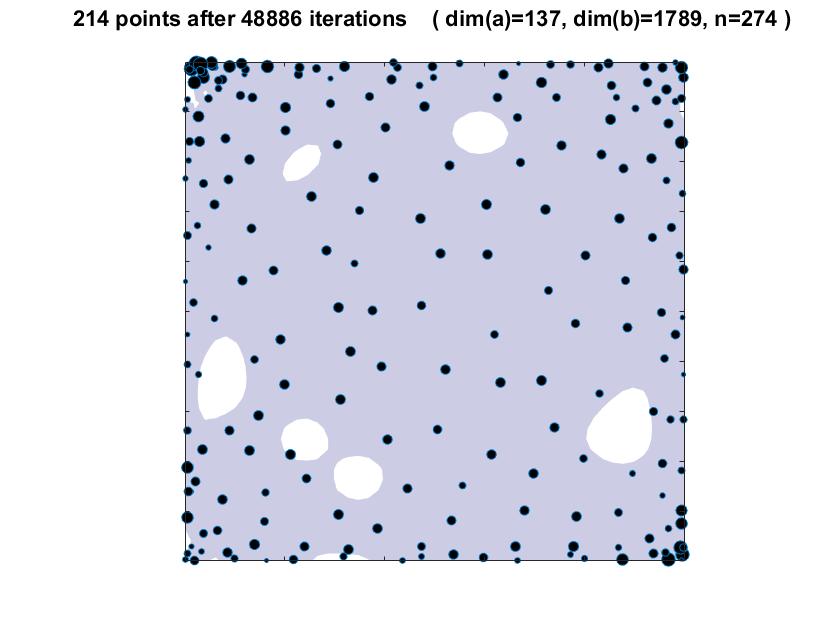} 
\hspace{-5mm}
\vspace{-3mm}
\caption{Points generated according to the algorithm for the Chebyshev basis for isotropic smoothness 
($R=1000$, $R'=15000$), with the area of points that would be rejected next (grey).}
\label{fig:cheb-new}
\end{figure}
\end{minipage}

\begin{minipage}{.95\linewidth}
\begin{figure}[H]
\includegraphics[clip, width=.5\linewidth, trim=30mm 20mm 20mm 0mm]{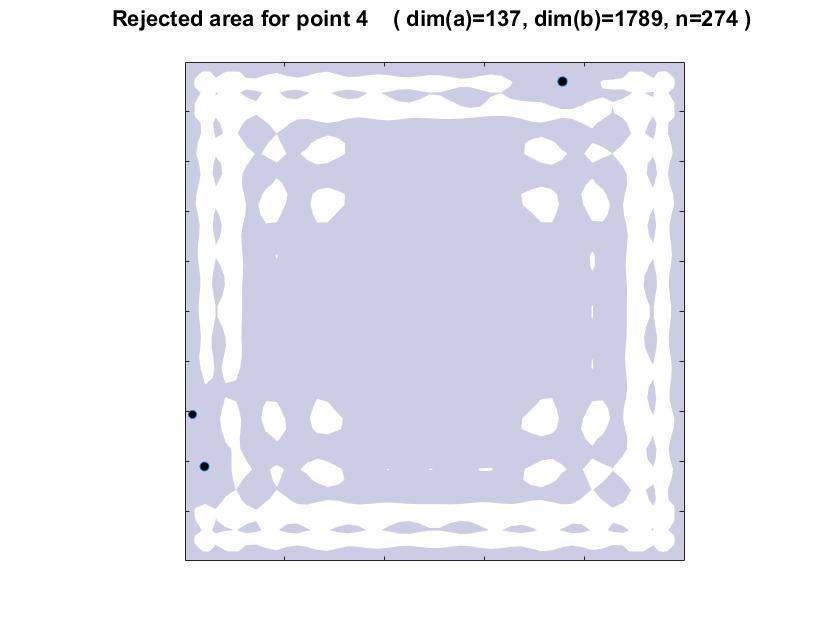} 
\includegraphics[clip, width=.5\linewidth, trim=30mm 20mm 20mm 0mm]{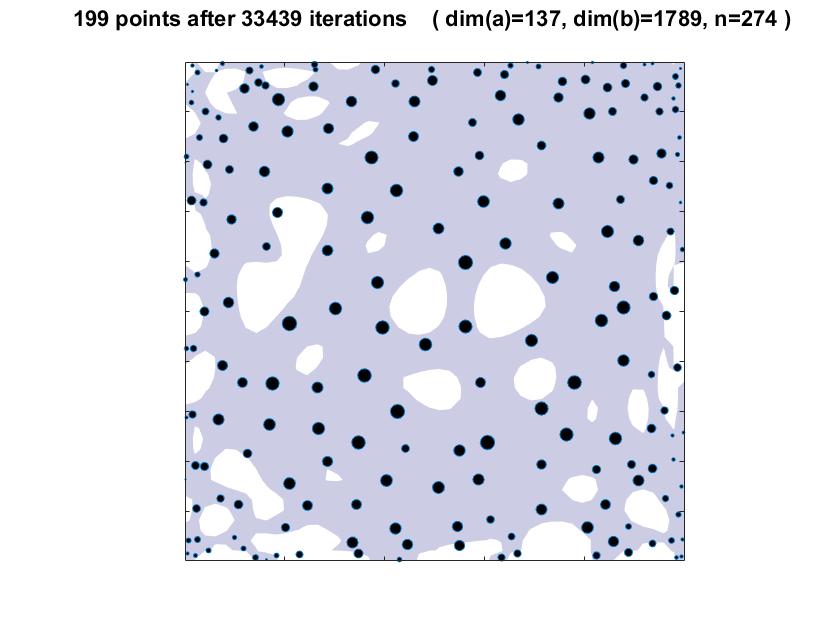} 
\hspace{-5mm}
\vspace{-3mm}
\caption{Points generated according to the algorithm for the Legendre basis for isotropic smoothness 
($R=1000$, $R'=15000$), with the area of points that would be rejected next (grey).}
\label{fig:legendre-new}
\end{figure}
\end{minipage}

\end{itemize}


\bigskip

\noindent
A.C., Mohammed VI Polytechnic University; \texttt{abdellah.chkifa@um6p.ma};\\
M.D., Brown University; \texttt{matthieu\_dolbeault@brown.edu};\\
D.K., University of Passau; \texttt{david.krieg@uni-passau.de};\\
M.U., Johannes Kepler University Linz; \texttt{mario.ullrich@jku.at}

\end{document}